\documentclass[11pt]{article}
\usepackage{caption}
\usepackage{amsthm}
\usepackage{subcaption}
\usepackage{cite}
\usepackage{graphics}
\usepackage{graphicx}
\graphicspath{{Photo/}}
\usepackage{authblk}
\usepackage{amsmath,amssymb,amsfonts}
\usepackage[english]{babel}
\usepackage{hyperref}
\usepackage{array}
\usepackage[rightcaption]{sidecap}
\usepackage{longtable}
\usepackage{enumitem}
\usepackage{scalerel}

\usepackage{graphicx}
\usepackage{color}
\usepackage{algorithm}
\usepackage{algorithmic}
\usepackage{comment}
\numberwithin{equation}{section}
\definecolor{darkgreen}{rgb}{0,0.7,0}

\newcommand{\RR}{\mathbb{R}}

\newcommand{\norm}[1]{\Vert #1 \Vert}

\newcolumntype{M}[1]{>{\centering\arraybackslash}m{#1}}

\usepackage[english]{babel}
\providecommand{\keywords}[1]{\textbf{{Keywords:}} #1}
\numberwithin{equation}{section}
\usepackage{color}

\usepackage[english]{babel}

	\newtheorem{remark}{Remark}[section]

	\newtheorem{theorem}{Theorem}[section]

	\newtheorem{lemma}{Lemma}[section]
	\newtheorem{proposition}{Proposition}[section]

\title{Feedback stabilization and observer design for sterile insect technique model}
\author[]{Kala Agbo bidi}
\affil[]{Sorbonne Universit\'{e}, CNRS, Universit\'{e} Paris-Cit\'{e}, Laboratoire Jacques-Louis Lions, LJLL,  INRIA  \'{e}quipe CAGE, F-75005 Paris, France,
\texttt{kala.agbo\_bidi@sorbonne-universite.fr}}

\date{\empty}

\begin{document}
\maketitle
\begin{abstract}
This paper focuses on the feedback global stabilization and observer construction for a sterile insect technique model. The sterile insect technique (SIT) is one of the most ecological methods for controlling insect pests responsible for worldwide crop destruction and disease transmission. In this work, we construct a feedback law that globally asymptotically stabilizes a SIT model at extinction equilibrium. Since the application of this type of control requires the measurement of different states of the target insect population, and, in practice, some states are more difficult or more expensive to measure than others, it is important to know how to construct a state estimator which from a few well-chosen measured states, estimates the other ones, as the one we build in the second part of our work. In the last part of our work, we show that we can apply the feedback control with estimated states to stabilize the full system.
\end{abstract}

\keywords{sterile insect technique; pest control; feedback control design; observer design; Lyapunov stability; mosquito population control; vector-borne disease}

\section{Introduction}	
SIT is presently one of the most ecological methods for controlling insect pests responsible for disease transmission or crop destruction worldwide. This technique consists in releasing sterile males into the insect pest population \cite{barclay1980sterile, gato2021sterile,vreysen2006sterile}. This approach aims at
 reducing fertility and, consequently, reducing the target insect population after a few generations. Classical SIT has been modeled and studied theoretically in a large number of papers to derive results to study the success of these strategies using discrete, continuous, or hybrid modeling approaches (for instance, the recent papers\cite{strugarek2019use,almeida2022optimal,almeidaBarrier2022,almeidaRolling2023,leculierRolling2023,bliman2022robust,bliman2019implementation}).

Despite this extensive  research, little has been done concerning the stabilization of the target population near extinction after the decay caused by the massive initial SIT intervention and  there are  still major difficulties due to the complexity of the dependency on climate, landscape and many
 other parameters which would be difficult to be
 integrated into the mathematical models studied.
 Not being able to consider all these
 parameters in our mathematical models and knowing
 that these external factors strongly impact the
 evolution of the density of the target population,
 we focus our studies on releases that now depend
 on the target population density measurements since, as we will see below, this makes our control more robust. Indeed, several monitoring
 tools can provide information on the size of the
 wild population throughout the year. So, a control that considers this
 information to adapt the size of the
 releases is possible and useful. This was already the case of  \cite{bidi2023global,bliman2022robust} in which a state feedback control law gives significant robustness qualities to the mathematical model of SIT.
Although this approach provides evidence in terms
of robustness because the control is directly
 adjusted according to the density of the population,
 its application requires to  continuously measure the
  different states of the model.
  In practice, traps allow data to
   be collected to analyze the control's impact and  technology is being developed that may allow us to obtain
   continuous data in the near future.

  However, specific categories of data are still problematic or very expensive to obtain. For example, during a SIT intervention, it is difficult to measure the density of young females that have not yet been fecundated or of females that were fecundated by  wild males.
  In this work we use another
    control theory tool, which consists of constructing
    a state estimator for a dynamical system   and using this
   estimator to apply feedback control.
A state observer or state estimator is a system that
 provides an estimate of the natural state using some partial measurements of the real system. In our case, using traps, wild males as well as sterile males, can be measured. Using the observer system technique, we have built a system that allows us to estimate all other states.
	
	The problem of observer design for linear systems was
	established and solved by \cite{kalman1960new}
	and  \cite{luenberger1964observing}.
	While Kalman's Observer \cite{kalman1960new} was highly successful for linear systems,
	extending it to nonlinear systems took a
	 lot of work. In several cases, the observer can be obtained from the
	 extended Kalman filter by a particular
	 choice of the matrix gain using linear matrix inequalities (LMIs).   The development of the observer in this paper
	 was motivated by its application to the SIT model.
	 A model of this process can be written as
 \begin{align}
	&\dot x = Ax + B(y)x + Du,\\& y = Cx,
 \end{align}
  where $y\in\RR^m$ is the output, $x\in\RR^n$ is the state vector,
  and $u\in\RR^p$ is the input. The output matrix $B(y)$ is such that the coefficients $b(y)_{ij}$ are bounded
  for all $i,j$.
  
  Our paper has three parts.  In the first part,
  thanks to the backstepping approach,
   we build a feedback control law that
   stabilizes the zero population state for the SIT model for the mosquito
   population, which considers only the compartments
   of young females and fertilized females presented
   in \cite{anguelov2012mathematical}.
   In the second part  we construct a
   state estimator for the SIT model. Finally, in the third part  we show that the application
   of this feedback, depending on the 
   measured states and the ones estimated thanks to the
   state estimator, globally stabilizes the system.
\section{Mosquito population dynamics}
The mosquito life cycle has several phases.
The aquatic stage comprises eggs, larvae, and pupa,
followed by the adult stage, where we consider both wild males and females. After
 emergence from the pupa, a female mosquito
 needs to mate and  then to get a blood meal
 before it  can start laying eggs. Then every $4-5$
days, it will take a blood meal and lay $100-150$
eggs at different places ($10-15$ per place).
For the mathematical description, we will consider the following compartments \cite{anguelov2012mathematical}.
\begin{itemize}
	\item $E$ the density of population in aquatic stage,
	\item $Y$ the density of young females, not yet laying eggs,
	\item $F$ the density of fertilized and egg-laying females,
	\item $M$ the density of males,
	\item $M_s$ the density of sterile males,
	\item $U$ the density of females that mate with sterile males.
\end{itemize}
The $Y$ compartment represents
the stage of the young females before
the start of their gonotropic cycle, i.e.,
 before they mate and take their first blood meal.
 It  generally lasts for 3 to 4 days.
The sterile insect technique introduces male
mosquitoes to compete with wild males. We denote by $M_s$ the density of sterile
mosquitoes and by $U$ the density of females that have mated with them.
We assume that a female mating mosquito has probability
 $\frac{M}{M+M_s}$ to mate with a wild male and probability $\frac{M_s}{M+M_s}$
 to mate with a sterile one. Hence, the transfer rate $\eta$ from the compartment $Y$
  splits into transfer rate of $\frac{\eta_1 M}{M+M_s}$ to compartment $F$ and a transfer rate of
  $\frac{\eta_2 M_s}{M+M_s}$ to compartment $U$ of females that will be laying sterile (nonhatching) eggs.
  The mathematical model is the system of ordinary differential equations presented in \cite{esteva2005mathematical}
  \begin{align}
	  &\dot E = \beta_E F(1-\frac{E}{K})  -(\delta_E + \nu_E)E,\label{eq:S1E1}\\
	  &\dot M = (1-\nu)\nu_E E - \delta_MM,\label{eq:S1E2}\\
	  &\dot Y = \nu\nu_E E - \frac{\eta_1 M }{M+M_s}Y-\frac{\eta_2 M_s }{M+M_s}Y-\delta_YY,\label{eq:S1E3}\\
	  &\dot F = \frac{\eta_1 M }{M+M_s} Y -\delta_FF,\label{eq:S1E4}\\
	  &\dot U = \frac{\eta_2 M_s }{M+M_s}Y -\delta_UU,\label{eq:S1E5}\\
	  &\dot M_s = u -\delta_sM_s.\label{eq:S1E6}
  \end{align}
  The  parameter $\delta_Y$
is the mortality rate, for young females (they can  die without mating for diverse reason like predators or other hostile environmental conditions).
Male mosquitoes can mate for most
of their lives. A female mosquito needs
a successful mating to be able to reproduce for the rest
of her life, $\beta_E>0$ is the oviposition rate;
$\delta_E,\delta_M,\delta_F,\delta_Y, \delta_s >0 $ are the death
rates  respectively for eggs, wild adult males, fertilized
females, young females and sterile males;
$\nu_E>0$ is the hatching rate for eggs;
$\nu\in (0,1)$, the probability
that a pupa gives rise to a
female, and $(1-\nu)$ is, therefore,
the probability of giving rise to a male.
$K>0$ is the environmental capacity
 for eggs. It can be interpreted as
 the maximum density of eggs that females
 can lay in breeding sites. Since here the
 larval and pupal compartments are not present,
 we consider that $E$ represents all the aquatic
 compartments, in which case,  this term $K$
 represents a logistic law's carrying capacity
 for the aquatic phase, which also includes the
 effects of competition between larvae.
  The control function $u$ represents the number of
  mosquitoes released during the SIT intervention.
  It is interesting to follow the evolution of the state $U$ because female mosquitoes, once fertilized by sterile males, will continue their gonotrophic cycle normally and, therefore, can still transmit disease.
  We will assume in this work that
  \begin{gather}\label{deltas>deltaM}
	\delta_s\geq\delta_M.
  \end{gather}

In \cite{anguelov2012mathematical,esteva2005mathematical},  equilibria and their stability property were studied for the system without control.
\begin{align}
	&\dot E = \beta_E F(1-\frac{E}{K})  -(\delta_E + \nu_E)E,\label{eq:SS11E1}\\
	&\dot M = (1-\nu)\nu_E E - \delta_MM,\label{eq:SS11E2}\\
	&\dot Y = \nu\nu_E E - (\eta_1+\delta_Y)Y,\label{eq:SS11E3}\\
    &\dot F = \eta_1 Y -\delta_FF.\label{eq:SS11E4}\\
\end{align}
Its basic offspring number is $\mathcal{R}_0 = \frac{\eta_1\beta_E\nu\nu_E}{\delta_F(\nu_E+\delta_E)(\eta_1+\delta_Y)}$.
For the rest of our work, we assume that
\begin{gather}\label{eq:biologicalcond}
    \mathcal{R}_0>1.
\end{gather}

\section{Global stabilization by a feedback law}
We assume that wild males are more likely to fertilize young females because they are born in the same egg-laying site. We define
\begin{gather}\label{eq:eta1-eta2}
	\Delta\eta = \eta_1-\eta_2\geq 0.
\end{gather}

Other authors, such as in
\cite{anguelov2012mathematical},
have already studied the
stability of this type of model.
The difference in our approach
lies in the kind of control
used initially for global stabilization.
Indeed, in most of the prior studies
the controls $u$ studied were
independent of system states. Some previous works have considered certain simple applications of feedback control to SIT (see, for instance, \cite{bliman2022robust, bliman2019feedback,2023Rossi}).
In a previous paper \cite{bidi2023global},
we used the backstepping method to build a
feedback control system that simplifies the SIT
model, which is presented in \cite{strugarek2019use},
assuming that all females are immediately fertilized.
Here we consider the system
\begin{align}
	&\dot E = \beta_E F(1-\frac{E}{K})  -(\delta_E + \nu_E)E,\label{eq:S11E1}\\
	&\dot M = (1-\nu)\nu_E E - \delta_MM,\label{eq:S11E2}\\
	&\dot Y = \nu\nu_E E -\frac{\Delta\eta M }{M+M_s}Y - (\eta_2+\delta_Y)Y,\label{eq:S11E3}\\
	&\dot F = \frac{\eta_1 M }{M+M_s} Y -\delta_FF,\label{eq:S11E4}\\
	&\dot U = \frac{\eta_2 M_s }{M+M_s}Y -\delta_UU,\label{eq:S11E5}\\
	&\dot M_s = u -\delta_sM_s.\label{eq:S11E6}
\end{align}
Let $\mathcal{N} := [0,+\infty)^6$ and $\mathcal{X} := (E,M,Y,F, U,M_s)^T$. When applying a feedback law $u:\mathcal{N}\rightarrow [0,+\infty)$, the closed-loop system is  the  system
\begin{equation}
	\dot{\mathcal{X}}= H(\mathcal{X},u(\mathcal{X})),
			\label{eq:closed-loop}
\end{equation}
where $H$ is the right-hand side of  Eqs \eqref{eq:S11E1}--\eqref{eq:S11E6}.
The construction method remains
the same as in our previous paper \cite{bidi2023global}.
In this work, we also consider solutions
in the Filippov sense
of our discontinuous closed-loop system (see, for instance \cite{1960-Filippov-MS,1967-Hermes-AP,1988-Filippov-book, 1994-Coron-Rosier-JMSEC,1998-Clarke-Ledyaev-Stern-1998,2005-Bacciotti-Rosier-book}  ).
Let  us define $x: = (E,M,Y,F,U)^T$. We must rewrite the
 target system \eqref{eq:S11E1}--\eqref{eq:S11E5}  in the following form to apply the backstepping method (see, for instance\cite[Theorem 12.24, page 334]{coron2007control}):
	\begin{equation}
		\left\{
		\begin{aligned}
			&\dot{x}= f(x,M_s),\\
			&\dot{M_s}  = u-\delta_sM_s,
		\end{aligned}
		\right.
	\end{equation}
	where $f:\RR^6\to\RR^5$ represents the right hand side of  \eqref{eq:S11E1}--\eqref{eq:S11E5}.
	We then consider the control system $\dot{x}= f(x,M_s)$ with
 the state being $x\in\mathcal{D}:=[0, +\infty)^5$ and
 the control being  $M_s\in [0,+\infty)$.
 We assume that  $M_s$ is of the form $M_s = \theta M$ for a constant $\theta >0$. Then, we  define
 and  study the closed-loop system
\begin{align}
	\dot{x}= f(x,\theta M).
	\label{eq:backeq}
\end{align}
Its offspring number is
\begin{gather}
   \label{eq:Rtheta}
	\mathcal{R}(\theta):= \frac{\beta_E\eta_1\nu\nu_E}{\delta_F(\nu_E+\delta_E)(\Delta\eta +(1+\theta)(\eta_2+\delta_Y))}.
\end{gather}
Note that if $\mathcal{R}(\theta)\leq 1$, $\textbf{0} \in \RR^5$ is the only equilibrium point of the system in $\mathcal{D}$.
Our next proposition shows that the feedback law $M_s=\theta M$ stabilizes our control system $\dot{x}= f(x, M_s)$ if $\mathcal{R}(\theta)<1$.

\begin{proposition}
\label{prop-case-thetaM-GAS}
	Assume that
\begin{align}
	\mathcal{R}(\theta)<1 \label{eq:inequalitytheta}.
\end{align} 
Then $\textbf{0}$  is globally exponentially
 stable in $\mathcal{D}$ for  system \eqref{eq:backeq}. The  exponential convergence rate is  bounded from above
 by the positive constant $c$  defined by relation \eqref{def-c>0}.
\end{proposition}
{\bf Proof.}
We apply Lyapunov's second theorem. To do so, we define
 $V: [0,+\infty)^5\rightarrow\RR_+$, $x\mapsto V(x)$,
\begin{equation}
\label{def-V}
	V(x) := \frac{(1+2\mathcal{R}(\theta))\nu\nu_E}{(\nu_E+\delta_E)(1-\mathcal{R}(\theta))} E
	 + \nu M + \frac{3\mathcal{R}(\theta)}{(1-\mathcal{R}(\theta))}Y + \frac{(2+\mathcal{R}(\theta))\beta_E\nu\nu_E}{\delta_F(\nu_E+\delta_E)(1-\mathcal{R}(\theta))} F + \sigma U,
\end{equation}
where $\sigma>0$ is a constant, that we will choose later.

\noindent As \eqref{eq:inequalitytheta} holds,
$ V$ is of class $\mathcal{C}^1$,
$V(x)>V((0,0,0,0,0)^T)=0, \;\forall x \in [0, +\infty)^5\setminus\{(0,0,0,0,0)^T\},$
$V(x)\to +\infty$  when $\norm{x}\to +\infty$ with $x\in\mathcal{D}$
and
\begin{multline*}
	\dot{V}(x) = -\frac{\beta_E\nu\nu_E}{(\nu_E+\delta_E)} F-\nu\delta_M M
	-\frac{(1+2\mathcal{R}(\theta))\nu\nu_E}{(\nu_E+\delta_E)(1-\mathcal{R}(\theta))} \frac{\beta_E}{K} FE \\
	- \nu^2\nu_EE-\frac{\eta_1\beta_E\nu\nu_E}{\delta_F(\nu_E+\delta_E)(1+\theta)}Y-\frac{\sigma\eta_2}{1+\theta} Y +\sigma\eta_2 Y-\sigma\delta_UU.
\end{multline*}
By choosing 
\begin{equation}
\label{def-sigma}
\sigma := \frac{\eta_1\beta_E\nu\nu_E\mathcal{R}(\theta)}{(1+\theta)\eta_2(\nu_E+\delta_E)\delta_F}
\end{equation}
we get
\begin{multline*}
	\dot{V}(x) = -\frac{\beta_E\nu\nu_E}{(\nu_E+\delta_E)} F-\nu\delta_M M
	-\frac{(1+2\mathcal{R}(\theta))\nu\nu_E}{(\nu_E+\delta_E)(1-\mathcal{R}(\theta))} \frac{\beta_E}{K} FE \\
	- \nu^2\nu_EE-\frac{\eta_1\beta_E\nu\nu_E(1+\theta(1-\mathcal{R}(\theta)))}{\delta_F(\nu_E+\delta_E)(1+\theta)^2}Y-\sigma\delta_UU.
\end{multline*}
and using once more \eqref{eq:inequalitytheta}, we get
\begin{align}
\label{dotV<0}
    \dot{V}(x) \leq -c V(x), \; \forall x \in [0, +\infty)^5,
\end{align}
with
\begin{multline}
\label{def-c>0}
c:=\min\left\{\frac{\nu(\nu_E+\delta_E)(1-\mathcal{R}(\theta))}{(1+2\mathcal{R}(\theta))},\delta_M,
\frac{\delta_F(1-\mathcal{R}(\theta))}{2+\mathcal{R}(\theta)}, \delta_U,\right.
\\
\left.
\frac{\eta_1\beta_E\nu\nu_E(1+\theta(1-\mathcal{R}(\theta)))}{\delta_F(\nu_E+\delta_E)(1+\theta)^2}\frac{(1-\mathcal{R}(\theta))}{3\mathcal{R}(\theta)} \right\}>0.
\end{multline}
This concludes the proof of Proposition~\ref{prop-case-thetaM-GAS}.
\hfill $\Box$\\
\begin{remark}
	 When the Allee effect is included in 
	the model (for instance \cite[Eq (2.5), Page 25]{strugarek2019use}), 
	the control $M_s = \theta M$ can still be used,
	 and the proof of the stability result can still be done using 
	 the same Lyapunov function \eqref{def-V}.
\end{remark}
 We  define
\begin{align}
	&\phi := \frac{(2+\mathcal{R}(\theta))\eta_1\beta_E\nu\nu_E-3\mathcal{R}(\theta)\Delta\eta\delta_F(\nu_E+\delta_E)}{\delta_F(\nu_E+\delta_E)(1-\mathcal{R}(\theta))(1+\theta)} - \frac{\eta_1\beta_E\nu\nu_E\mathcal{R}(\theta)}{(1+\theta)^2(\delta_E+\nu_E)\delta_F},\\&
Q := 3(\eta_2+\delta_Y)(1+\theta)(\nu_E+\delta_E)\delta_F-(1-\mathcal{R}(\theta))\eta_1\beta_E\nu\nu_E,
\end{align}
and for $\alpha>0$,  the map $G: \mathcal{N}:=[0,+\infty)^6\rightarrow \RR$, $(x^T,M_s)^T\mapsto G((x^T,M_s)^T)$ by
\begin{multline}
		G((x^T,M_s)^T):=	\frac{\phi Y(\theta M + M_s)^2}{\alpha(M+M_s)(3\theta M + M_s)}
+ \frac{((1-\nu)\nu_E\theta E -\theta \delta_M M)(\theta M +3M_s)}{3\theta M + M_s}
\\ +\delta_sM_s + \frac{1}{\alpha}(\theta M-M_s) \text{ if } M+M_s\not=0,
\end{multline}
\begin{equation}
G((x^T,M_s)^T):=0 \text{ if } M+M_s=0.
\end{equation}

 Finally, let us define the feedback law $u: \mathcal{N}\rightarrow [0,+\infty)$,
 $(x^T,M_s)^T\mapsto u((x^T,M_s)^T)$, by
\begin{align}
		u((x^T,M_s)^T):=\max\left(0,G((x^T,M_s)^T)\right). \label{eq:backcontr}
\end{align}
The global stability result is the following.	
\begin{theorem}
	\label{thm-backstepping}
	Assume that \eqref{eq:inequalitytheta} holds.
	Then $\textbf{0}\in \mathcal{N}$ is globally 
	exponentially stable in $\mathcal{N}$ 
	for system \eqref{eq:S11E1}--\eqref{eq:S11E5}  
	with the feedback law \eqref{eq:backcontr}. The exponential convergence rate is bounded by the
	positive constant  $c_p$ defined by
	\begin{gather}
		c_p: = \min\{c, \frac{1}{\alpha},\delta_M,\frac{Q}{3(1+\theta)\delta_F(\nu_E+\delta_E)},\delta_U\}.
	\end{gather}
\end{theorem}
\begin{lemma}\label{see:Lemmaphi}
Assume that \eqref{eq:biologicalcond}  and \eqref{eq:inequalitytheta} hold, then $\phi >0$.
\end{lemma}
{\bf Proof.}\\
Let us define $\phi_1  := \frac{(2+\mathcal{R}(\theta))\eta_1\beta_E\nu\nu_E-3\mathcal{R}(\theta)\Delta\eta\delta_F(\nu_E+\delta_E)}{\delta_F(\nu_E+\delta_E)(1-\mathcal{R}(\theta))(1+\theta)} $. We get from the relation \eqref{eq:biologicalcond} that $\eta_1\beta_E\nu\nu_E>\delta_F(\nu_E+\delta_E)(\eta_1+\delta_Y)$. So
\begin{gather}
    \phi_1>\frac{2\eta_1}{(1+\theta)} + \frac{(2+\mathcal{R}(\theta))\delta_Y + 3\mathcal{R}(\theta)\eta_2}{(1-\mathcal{R}(\theta))(1+\theta)}.
\end{gather}
From  relation \eqref{eq:inequalitytheta} we get
$\frac{\beta_E\eta_1\nu\nu_E}{\delta_F(\nu_E+\delta_E)}<\Delta\eta +(1+\theta)(\eta_2+\delta_Y)$. Thus
\begin{align*}
     \phi&>\frac{2\eta_1}{(1+\theta)} + \frac{(2+\mathcal{R}(\theta))\delta_Y + 3\mathcal{R}(\theta)\eta_2}{(1-\mathcal{R}(\theta))(1+\theta)} - \frac{\Delta\eta\mathcal{R}(\theta) +(1+\theta)(\eta_2+\delta_Y)\mathcal{R}(\theta)}{(1+\theta)^2},\\&>\frac{2\eta_1\mathcal{R}(\theta)}{(1+\theta)} + \frac{(2+\mathcal{R}(\theta))\delta_Y + 3\mathcal{R}(\theta)\eta_2}{(1-\mathcal{R}(\theta))(1+\theta)} +\frac{\eta_2\mathcal{R}(\theta) }{(1+\theta)^2}- \frac{\eta_1\mathcal{R}(\theta)}{(1+\theta)^2}-\frac{(\eta_2+\delta_Y)\mathcal{R}(\theta)}{1+\theta},\\
     &>\frac{\eta_1\mathcal{R}(\theta)(1+2\theta)}{(1+\theta)^2} + \frac{2\mathcal{R}(\theta)\eta_2 + 2\delta_Y + \mathcal{R}(\theta)^2(\eta_2+\delta_Y)}{(1-\mathcal{R}(\theta))(1+\theta)} +\frac{\eta_2\mathcal{R}(\theta) }{(1+\theta)^2},\\
     &>0.
\end{align*}

\hfill $\Box$\\
{\bf Proof of Theorem \ref{thm-backstepping}.}
Let  $\alpha>0$ and define $W:\mathcal{N}\rightarrow \RR$ by
\begin{gather}
\label{eq:lyapunov-functionW}
	W((x^T,M_s)^T) :=  V(x) + {\alpha}\frac{(\theta M-M_s)^2}{\theta M + M_s} \text{ if } M+M_s\not = 0,
\\
\label{eq:lyapunov-functionW-M=Ms=0}
W((x^T,M_s)^T) := V(x)  \text{ if } M+M_s = 0.
\end{gather}
We have
\begin{gather}
W \text{ is continuous},
\\
W \text{ is of class $\mathcal{C}^1$ on $\mathcal{N}\setminus \left\{(E,M,Y,F,U, M_s)^T\in \mathcal{N};\; M+M_s=0\right\}$},
\\
\label{Winftyatinfty}
W((x^T,M_s)^T)\to +\infty \;\;\mbox{as} \;\; \norm{x} +M_s\to +\infty,\text{ with $x\in \mathcal{D}$ and $M_s\in [0,+\infty)$},
\\
\label{W>0}
W((x^T,M_s)^T)>W(\textbf{0})=0, \; \forall (x^T,M_s)^T\in \mathcal{N}\setminus\{\textbf{0}\}.
\end{gather}
 From now on, and until the end of this proof, we assume that $(x^T,M_s)^T$ is in $\mathcal{N}$ and until \eqref{dotW<-case-2-2} below we further assume that
\begin{gather}
\label{MMsnotboth01}
(M,M_s)\not =(0,0).
\end{gather}
One has
\begin{equation*}
\begin{array}{rcl}
	\dot{W}((x^T,M_s)^T) &= & \nabla V(x)^T\cdot f(x, M_s)
+\alpha(\theta M-M_s)\frac{2(\theta \dot{M}-\dot{M}_s)(\theta M + M_s)-(\theta\dot{M}+\dot{M_s})(\theta M-M_s)}{(\theta M + M_s)^2}, \\ &= &\nabla V(x)^T\cdot f(x, \theta M) + \nabla V(x)^T\cdot(f(x,M_s)-f(x,\theta M))
\\ && + \alpha(\theta M-M_s)\frac{\theta\dot{M}(\theta M + 3M_s)-\dot{M}_s(3\theta M+M_s)}{(\theta M + M_s)^2}.
\end{array}
\end{equation*}
Since
\begin{equation*}
\begin{array}{rcl}
	\nabla V(x)^T\cdot (f(x,M_s)-f(x,\theta M))& =& \begin{pmatrix}\frac{(1+2\mathcal{R}(\theta))\nu\nu_E}{(\nu_E+\delta_E)(1-\mathcal{R}(\theta))}\\
	  \nu \\  \frac{3\mathcal{R}(\theta)}{(1-\mathcal{R}(\theta))}\\  \frac{(2+\mathcal{R}(\theta))\beta_E\nu\nu_E}{\delta_F(\nu_E+\delta_E)(1-\mathcal{R}(\theta))} \\ \frac{\eta_1\beta_E\nu\nu_E\mathcal{R}(\theta)}{(1+\theta)\eta_2(\nu_E+\delta_E)\delta_F}
\end{pmatrix}
\cdot
\begin{pmatrix}
0\\ 0\\  -\frac{\Delta\eta ( \theta M-M_s)}{(M +M_s)(1+\theta)}Y\\\frac{\eta_1 ( \theta M-M_s)}{(M +M_s)(1+\theta)}Y\\-\frac{\eta_2 ( \theta M-M_s)}{(M +M_s)(1+\theta)}Y
\end{pmatrix}
\\
\\
&=&\displaystyle \frac{\phi Y (\theta M-M_s)}{M+M_s},
\end{array}
\end{equation*}
\begin{eqnarray}	
\dot{W}((x^T,M_s)^T)&= & \nabla V(x)^T\cdot f(x, \theta M) + \alpha\frac{(\theta M-M_s)}{(\theta M + M_s)^2}
\nonumber
\\&& \Big[\frac{(\nabla V(x)\cdot(f((x^T,M_s)^T)-f(x,\theta M)))(\theta M + M_s)^2}{\alpha(\theta M-M_s)}\nonumber
\\&& +\theta \dot{M}(\theta M +3M_s)-\dot{M}_s(3\theta M + M_s)\Big]\nonumber\\ &=&
\dot{V}(x) + \alpha\frac{(\theta M-M_s)}{(\theta M + M_s)^2}\Big[ \frac{\phi Y(\theta M + M_s)^2}{\alpha(M+M_s)}\\&& + ((1-\nu)\nu_E\theta E -\theta \delta_M M)(\theta M +3M_s)\nonumber
\\&&-u(3\theta M + M_s)+\delta_s{M}_s(3\theta M + M_s)\Big].
\label{eq:dotW=}
\end{eqnarray}
We take $u$ as given by \eqref{eq:backcontr}. Therefore, in the case where
\begin{multline}
\label{psigamma>}
	\frac{\phi Y (\theta M + M_s)^2}{\alpha(M+M_s)} + ((1-\nu)\nu_E\theta E -\theta \delta_M M)(\theta M +3M_s)
\\+
\delta_sM_s(3\theta M + M_s) + \frac{1}{\alpha}(\theta M-M_s)(3\theta M+M_s)>0,
\end{multline}
	\begin{multline*}
		u=\frac{1}{3\theta M + M_s}\Big[	\frac{\phi Y (\theta M + M_s)^2}{\alpha(M+M_s)} + ((1-\nu)\nu_E\theta E -\theta \delta_M M)(\theta M +3M_s)\\
+\delta_sM_s(3\theta M + M_s) + \frac{1}{\alpha}(\theta M-M_s)(3\theta M+M_s)\Big],
	\end{multline*}
which, together with \eqref{eq:dotW=}, leads to
\begin{align}
\label{dotW<-case-1}
	\dot{W}((x^T,M_s)^T)= \dot{V}(x) -\frac{(\theta M-M_s)^2(3\theta M+M_s)}{(\theta M + M_s)^2}.
\end{align}
Otherwise, i.e., if \eqref{psigamma>} does not hold,
\begin{multline}
	\frac{\phi Y (\theta M + M_s)^2}{\alpha(M+\gamma_sM_s)}  + ((1-\nu)\nu_E\theta E -\theta \delta_M M)(\theta M +3M_s) \\
+\delta_sM_s(3\theta M +M_s)
+ \frac{1}{\alpha}(\theta M-M_s)(3\theta M+M_s)\leq 0,\label{eq:inequalityueu}
\end{multline}
so, by \eqref{eq:backcontr},
\begin{gather}
\label{u=0}
u=0.
\end{gather}
We consider two cases.
If   $\theta M > M_s$  using \eqref{eq:dotW=}, \eqref{eq:inequalityueu} and \eqref{u=0}
\begin{align}
\label{dotW<-case-2-1}
	\dot{W}((x^T,M_s)^T)&\leq  \dot{V}(x) -\frac{(\theta M-M_s)^2(3\theta M+M_s)}{(\theta M + M_s)^2},\nonumber
    \\&\leq -c V(x) -\frac{(\theta M-M_s)^2}{\theta M + M_s},\nonumber
    \\&\leq- c_1 W((x^T,M_s)^T),
\end{align}
with
\begin{gather}\label{c1}
    c_1:=\min\{c, \frac{1}{\alpha}\}>0.
\end{gather}

Otherwise, ff $\theta M \leq  M_s$,
using once more \eqref{eq:dotW=} and \eqref{u=0}
\begin{multline}
\dot{W}((x^T,M_s)^T) =\dot{V}(x) + \alpha\frac{(\theta M-M_s)}{(\theta M + M_s)^2}\Big[ \frac{\phi Y(\theta M + M_s)^2}{\alpha(M+M_s)}  \\
+ \theta ((1-\nu)\nu_EE-\delta_M M)(\theta M +3M_s)\\+\delta_s{M}_s(3\theta M + M_s)\Big].
\end{multline}
\begin{multline}
\label{case-2-dotWdecom}
\dot{W}((x^T,M_s)^T) =\dot{V}(x) + \alpha\frac{(\theta M-M_s)}{(\theta M + M_s)^2}\Big[ \frac{\phi Y(\theta M + M_s)^2}{\alpha(M+M_s)}  
+ \theta ((1-\nu)\nu_EE\Big] \\+\alpha\frac{(\theta M-M_s)}{(\theta M + M_s)^2}\Big[ -\delta_M M(\theta M +3M_s)+\delta_s{M}_s(3\theta M + M_s)\Big].
\end{multline}
From  Lemma \ref{see:Lemmaphi} we deduce that $\phi>0$ and as $(x^T,M_s)^T\in \mathcal{N}$, one has  $\frac{\phi Y(\theta M + M_s)^2}{\alpha(M+M_s)}  
+ \theta (1-\nu)\nu_EE\geq 0$.
 \begin{eqnarray}
	\text{So},\;\;\theta M -M_s\leq 0\implies \alpha\frac{(\theta M-M_s)}{(\theta M + M_s)^2}\Big[ \frac{\phi Y(\theta M + M_s)^2}{\alpha(M+M_s)}  
	+ \theta (1-\nu)\nu_EE\Big]\leq 0.
\end{eqnarray}
Equation \eqref{case-2-dotWdecom} becomes
\begin{eqnarray}
\label{case-2-dotW}
\dot{W}((x^T,M_s)^T) \leq \dot{V}(x) +\alpha\frac{(\theta M-M_s)}{(\theta M + M_s)^2}\Big[ -\theta\delta_M M(\theta M +3M_s)+\delta_s{M}_s(3\theta M + M_s)\Big].
\end{eqnarray}
The inequality \eqref{deltas>deltaM} gives $\delta_s\geq \delta_M$ and one has
\begin{eqnarray}\label{deltaMthetaIneq}
	-\theta \delta_M M(\theta M +3M_s)+\delta_s{M}_s(3\theta M + M_s) &\geq& -\theta\delta_M M(\theta M +3M_s)+\delta_M{M}_s(3\theta M + M_s) \nonumber\\&\geq& \delta_M(M_s-\theta M)(M_s+\theta M).
\end{eqnarray}
 \eqref{deltaMthetaIneq} together with $\theta M-M_s\leq 0$, implies that
\begin{align}
\label{dotW<-case-2-2}
		\dot{W}((x^T,M_s)^T) &\leq \dot{V}(x) - \alpha\delta_M\frac{(\theta M-M_s)^2}{(\theta M + M_s)},\nonumber
   \\ &\leq -c_2 W((x^T,M_s)^T),
\end{align}
with
\begin{gather}\label{c2}
   c_2:=\min\{c, \delta_M\}>0.
\end{gather}
Let us now deal with the case where \eqref{MMsnotboth01} is not satisfied.
Note that, for every $\tau\geq0$,  $M(\tau)+M_s(\tau)>0 $ implies that $M(t)+M_s(t)>0$ for all
$t\geq \tau$. Thus, if  $M(0)+M_s(0)=0$,  there exists $t_s\in[0,+\infty]$ such that $M(t)+M_s(t)=0$ if and
only if $t\in [0,t_s]\setminus\{+\infty\}.$  Let us  study  only the case $t_s\in (0,+\infty)$ (the case $t_s=0$ is obvious and the case $t_s=+\infty$ is a corollary of our study of the case $t_s\in (0,+\infty)$).
Let us first point out that, for every $(M,M_s)^T\in [0,+\infty)^2$ such that $M+M_s>0$, one has
\begin{align*}
	&\frac{M}{M+M_s}\leq 1,\\ &(\theta M+M_s)^2\leq (3\theta M+M_s)^2\; \text{and}\; \frac{(\theta M+M_s)^2}{(M+M_s)(3\theta M+M_s)} \leq \frac{(3\theta M+M_s)}{M+M_s} \leq 3\theta +1, 
	\\&\frac{\theta M +3M_s}{3\theta M + M_s}= \frac{\theta M}{3\theta M + M_s}+\frac{3M_s}{3\theta M + M_s}\leq \frac{1}{3}+3\leq 4.
\end{align*}
So
\begin{equation}
\label{ineqMMs}
\frac{M }{M+M_s}\in [0,1], \;\frac{(\theta M +M_s)^2}{(M+M_s)(3\theta M+M_s)}\in [0, 3\theta +1], \text{  and }
\frac{\theta M +3M_s}{3\theta M + M_s}\in [0, 4].
\end{equation}

Let  $t\mapsto \mathcal{X}(t)=(E(t),M(t),Y(t),F(t), U(t), M_s(t))^T$  be a solution (in the Filippov sense)
of the closed-loop system \eqref{eq:S11E1}--\eqref{eq:S11E5} such that, for some  $t_s\in (0,+\infty)$
\begin{equation}
\label{M+Ms=0ts}
M(t)+M_s(t)=0,\; \forall t\in [0,t_s].
\end{equation}
Note that \eqref{M+Ms=0ts} implies that
\begin{equation}
\label{M=Ms=0ts}
M(t)=M_s(t)=0,\; \forall t\in [0,t_s].
\end{equation}

  From \eqref{ineqMMs}, \eqref{M=Ms=0ts} and the definition of a Filippov solution, one has on $(0,t_s)$
    \begin{align}
    \label{eq-Filippov}
        \left(\begin{array}{ccccc}
			\dot{E}\\\dot{{M}}\\\dot{{Y}}\\\dot{{F}}\\\dot{{U}}\\\dot{{M}}_s
		\end{array}\right) = \begin{pmatrix}
			\beta_E {F}(1-\frac{E}{K}) - \big( \nu_E + \delta_E \big) {E}
			\\  (1-\nu)\nu_E E - \delta_M M
			\\\nu\nu_E E-\kappa(t)\Delta\eta Y-(\eta_2+\delta_Y)Y
			\\  \eta_1 Y \kappa(t) - \delta_F {F}
             \\ \eta_2(1-\kappa(t)) Y - \delta_U U
			\\ {Yg_1(t)+Eg_2(t)} -\delta_s{M}_s
		\end{pmatrix}
    \end{align}
with
\begin{equation}
\kappa(t)\in [0,1],\; g_1(t)\in \frac{\phi}{\alpha}[0,3\theta+1] \text{ and } g_2(t)\in
	(1-\nu)\nu_E\theta [0,4].
\end{equation}
 From \eqref{M=Ms=0ts} and the second line of \eqref{eq-Filippov}, one has
\begin{equation}
\label{E=0ts}
E(t)=0, \; \forall t \in [0,t_s].
\end{equation}
 From the first line of \eqref{eq-Filippov} and  \eqref{E=0ts}, we get 
\begin{equation}
\label{F=0ts}
F(t)=0, \; \forall t \in [0,t_s].
\end{equation}

Let us first consider the case where $Y(0) = 0$. Then, from the third line of \eqref{eq-Filippov} and \eqref{E=0ts}, one has
\begin{equation}
\label{Y=0ts}
Y(t)=0, \; \forall t \in [0,t_s].
\end{equation} 
To summarize, from \eqref{M=Ms=0ts},  the fifth line of \eqref{eq-Filippov}, \eqref{E=0ts}, \eqref{F=0ts} and \eqref{Y=0ts}
\begin{equation}
\label{Five=0ts}
E(t)=M(t)=Y(t)=F(t)=M_s(t)=0 \text{ and }\dot U(t)=-\delta_UU(t),\;\forall t \in [0,t_s],
\end{equation} 
which, with \eqref{def-V}, \eqref{def-c>0} and \eqref{eq:lyapunov-functionW-M=Ms=0}, gives
\begin{gather}
\label{dotWleq-cW}
\dot W(t)=-\sigma\delta_UU(t)\leq -\delta_U W(t), \; \forall t \in [0,t_s].
\end{gather}

Let us finally consider the case where $Y(0) >0 $. Then, from the third line of \eqref{eq-Filippov},
\begin{equation}
\label{Y>0ts}
Y(t)>0, \; \forall t \in [0,t_s],
\end{equation} 
which, together with the fourth line of \eqref{eq-Filippov} and \eqref{F=0ts}, implies
\begin{equation}
\label{kappa0ts}
\kappa(t)=0, \; \forall t \in [0,t_s].
\end{equation} 
To summarize, from \eqref{M=Ms=0ts},  the third and the fifth line of \eqref{eq-Filippov}, \eqref{E=0ts}, \eqref{F=0ts} and \eqref{kappa0ts},
\begin{equation*}
\label{Four=0ts}
E(t)=M(t)=F(t)=M_s(t)=0, \; \dot Y(t)=-(\eta_2+\delta_Y)Y(t),\text{ and }\dot U(t)=\eta_2Y -\delta_UU(t),\;\forall t \in [0,t_s],
\end{equation*}
which, with \eqref{def-V}, \eqref{def-sigma}, \eqref{def-c>0} and \eqref{eq:lyapunov-functionW-M=Ms=0}, gives
\begin{gather}
\label{dotWleq-cW-last}
\begin{array}{rcl}
\dot W(t)&=&\displaystyle -(\eta_2+\delta_Y)\frac{3\mathcal{R}(\theta)}{(1-\mathcal{R}(\theta))}Y(t)+\eta_2\sigma Y(t)
-\sigma\delta_UU(t)
\\
&=&\displaystyle -\mathcal{R}(\theta)\left((\eta_2+\delta_Y)\frac{3}{(1-\mathcal{R}(\theta))}
-\frac{\eta_1\beta_E\nu\nu_E}{(1+\theta)(\nu_E+\delta_E)\delta_F}\right)Y(t)-\sigma\delta_UU(t)
\\&=& \displaystyle -\mathcal{R}(\theta)\left(\frac{Q}{(1-\mathcal{R}(\theta))(1+\theta)(\nu_E+\delta_E)\delta_F}
\right)Y(t)-\sigma\delta_UU(t)
\end{array}
\end{gather}
where 
\begin{gather}
    Q := 3(\eta_2+\delta_Y)(1+\theta)(\nu_E+\delta_E)\delta_F-(1-\mathcal{R}(\theta))\eta_1\beta_E\nu\nu_E.
\end{gather}
To end the proof we have to prove that $Q >0$.
Using the relation \eqref{eq:Rtheta}  and \eqref{eq:inequalitytheta} we have
\begin{gather}
	\beta_E\eta_1\nu\nu_E<\mathcal{R}(\theta) \delta_F(\nu_E+\delta_E)\Delta\eta  +  \delta_F(\nu_E+\delta_E)(1+\theta)(\eta_2+\delta_Y).
\end{gather}
Recall that  $\Delta\eta = \eta_1-\eta_2$. One has
\begin{align*}
    Q&= 3(\eta_2+\delta_Y)(1+\theta)(\nu_E+\delta_E)\delta_F-\eta_1\beta_E\nu\nu_E + \mathcal{R}(\theta)\eta_1\beta_E\nu\nu_E\\&> 2(\eta_2+\delta_Y)(1+\theta)(\nu_E+\delta_E)\delta_F-\mathcal{R}(\theta)\Delta\eta(\nu_E+\delta_E)\delta_F + \mathcal{R}(\theta)\eta_1\beta_E\nu\nu_E\\&> 2(\eta_2+\delta_Y)(1+\theta)(\nu_E+\delta_E)\delta_F-\mathcal{R}(\theta)\eta_1(\nu_E+\delta_E)\delta_F + \mathcal{R}(\theta)\eta_1\beta_E\nu\nu_E + \mathcal{R}(\theta)\eta_2(\nu_E+\delta_E)\delta_F.
\end{align*}
From the relation \eqref{eq:biologicalcond}, $\eta_1\beta_E\nu\nu_E>\delta_F(\nu_E+\delta_E)(\eta_1+\delta_Y)$.
\begin{align*}
    Q&> 2(\eta_2+\delta_Y)(1+\theta)(\nu_E+\delta_E)\delta_F-\mathcal{R}(\theta)\eta_1(\nu_E+\delta_E)\delta_F + \mathcal{R}(\theta)\delta_F(\nu_E+\delta_E)(\eta_1+\delta_Y) + \mathcal{R}(\theta)\eta_2(\nu_E+\delta_E)\delta_F\\&>2(\eta_2+\delta_Y)(1+\theta)(\nu_E+\delta_E)\delta_F+ \mathcal{R}(\theta)(\eta_2+\delta_Y)(\nu_E+\delta_E)\delta_F\\&>0.
\end{align*}
We get
\begin{gather}
\dot W(t)\leq -c'W(t), \; \forall t \in [0,t_s],
\end{gather}
where 
\begin{gather}\label{cprime}
    c': = \min\{\frac{Q}{3(1+\theta)\delta_F(\nu_E+\delta_E)},\delta_U\}.
\end{gather}
This proves  Theorem~\ref{thm-backstepping} and gives the global exponential stability.
 From \eqref{c1}, \eqref{c2} and \eqref{cprime} we obtain an estimate
on the exponential decay rate
\begin{gather}\label{cp}
	c_p: = \min\{c, \frac{1}{\alpha},\delta_M,\frac{Q}{3(1+\theta)\delta_F(\nu_E+\delta_E)},\delta_U\}.
\end{gather}
 \hfill $\Box$\\

\subsection{Numerical simulations}
Note that $\eta_1$ represents the natural fertility rate in the mosquito population. Wild males have a shorter maturity time in their life cycle than females. Thus, the fertilization phase is essentially   around the hatching site. Sterile males are artificially released into the intervention region. We denote by $p$ with  ($0\leq p\leq 1$)  the proportion of sterile males that are released. Also, the effective fertilization during the mating could be diminished due to the sterilization, which leads us to assume that the effective mating rate of sterile insects is given by $q\eta_1$ with $0\leq q\leq 1$. Putting together these assumptions we get that the probability for a young female to mate with sterile males is $\frac{\eta_2 M}{M+M_s}$ with
$\eta_2 = pq \eta_1$. For the numerical simulation we take $\eta_1 = 1$ and $\eta_2= 0.7$.
The numerical simulations of
the  dynamics when applying the
feedback  \eqref{eq:backcontr} is
 given in Figure~\ref{fig:simulation1}.
The parameters
we use are given in the following table.

\begin{table}[H]
\caption{Value for the parameters of system \eqref{eq:S11E1}--\eqref{eq:S11E4} (see \cite{anguelov2012mathematical,strugarek2019use}).
		 Units are days$^{-1}$ except for $\nu$.}
	\label{eq:tableparametre}
    \centering
	\setlength{\tabcolsep}{10mm}
	\begin{tabular}{lll}
		\hline
		Parameters  &  Description &Value  \\
		\hline
		$\beta_E$ & Effective fecundity &10\\
		
		$\nu_E$ & Hatching parameter & 0.05\\
		
		$\delta_E$& Mosquitoes in aquatic
		phase death rate &0.03\\
		
		$\delta_F$& Fertilized female death rate & 0.04\\
		
		$\delta_Y$& Young female death rate& 0.04\\
       
		$\delta_M$ & Male death rate &  0.1\\
		
		$\delta_s$ & Sterilized male death rate & 0.12\\
		
		$\nu$ & Probability of emergence & 0.49\\
		\hline
	\end{tabular}
\end{table}
With the parameters
given in Table \ref{eq:tableparametre},
 condition \eqref{eq:inequalitytheta} is $\theta >102,06$. We fix   $K = 21000$
 and we consider the persistence
  equilibrium $z_0 = (E^0, M^0, Y^0, F^0, U^0, M_s^0)$ as initial condition.
  That gives $E^0 = 20 700, M^0 = 5300, Y_0 = 1500, F^0 = 13000
$ and $U^0=M_s^0 = 0$.  We take $\theta =290$ and $\alpha =90$.

\begin{figure}[h]
	\centering
	\begin{subfigure}[H]{0.45\textwidth}
		\centering
     \includegraphics[width=\textwidth]{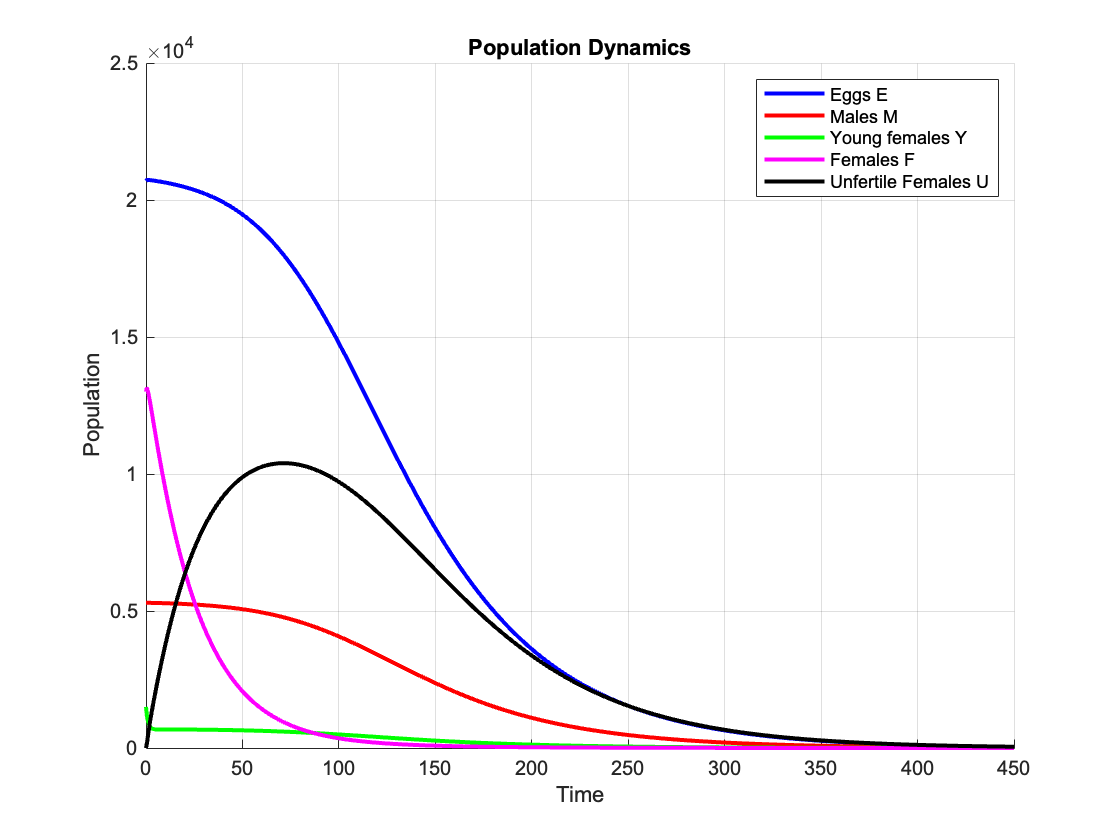}
	    \caption{Evolution of states $E$, $M$, $Y$, $F$ and $U$}  
		\label{fig:evolutionEMF}
	\end{subfigure}
	\hfill
	\begin{subfigure}[H]{0.45\textwidth}
		\centering
	     \includegraphics[width=\textwidth]{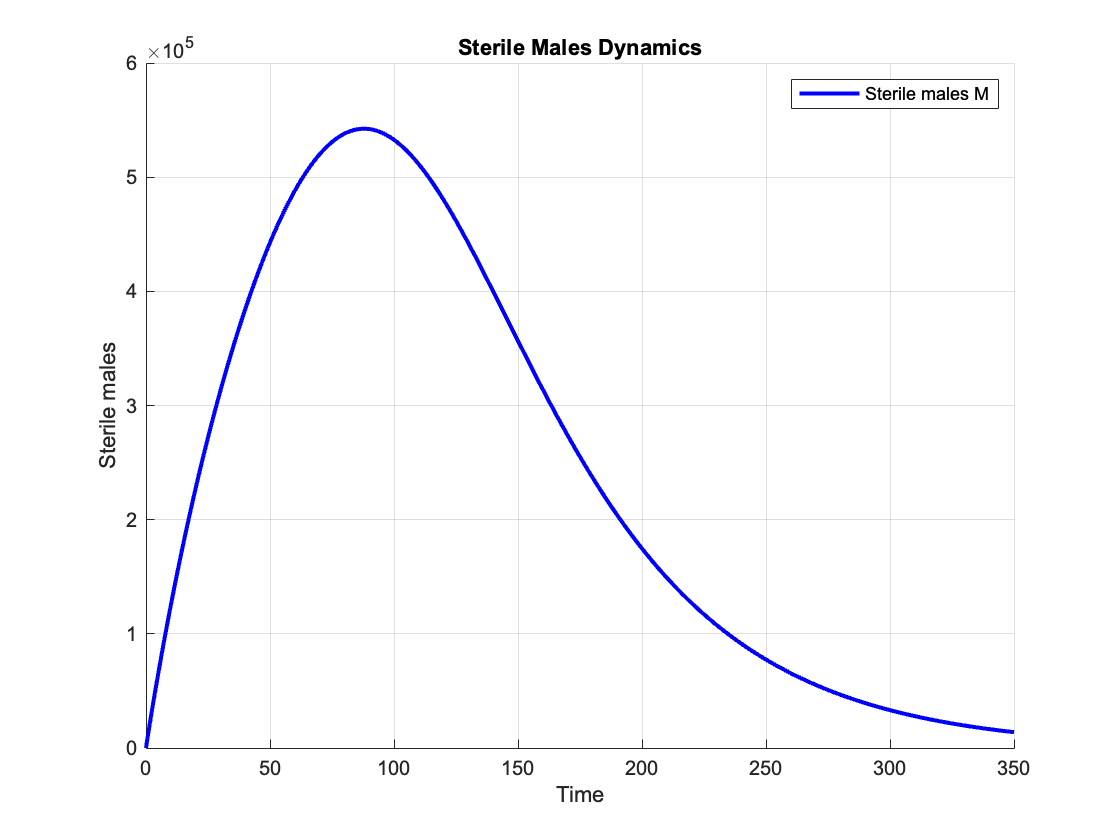}
	    \caption{Evolution of $M_s$ }
		\label{fig:EvolutionMs-1}
	\end{subfigure}
	\hfill
	\begin{subfigure}[H]{0.45\textwidth}
		\centering
		\includegraphics[width=\textwidth]{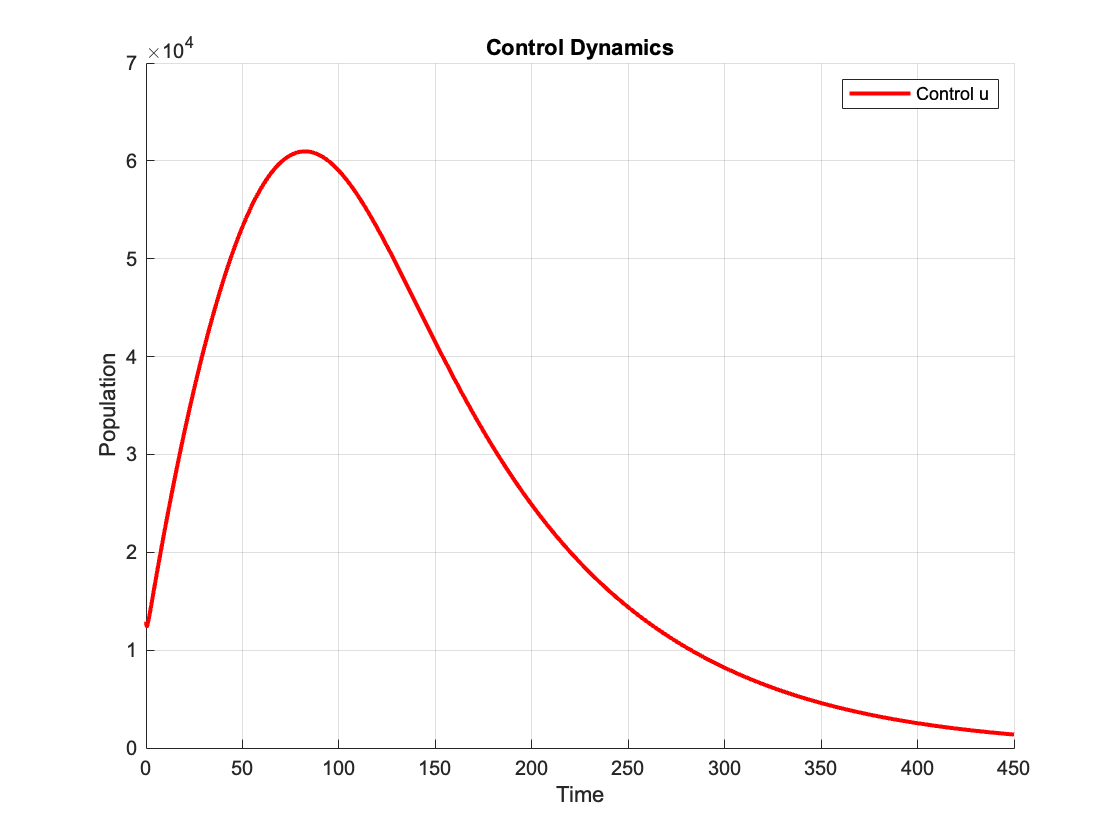}
	    \caption{Evolution of the control function $u$ }
		\label{fig:fiveoverx-1}
	\end{subfigure}
	\vspace{2mm}
	\caption{(a): Plot of $E,M,Y$, $F$ and $U$ when applying the feedback \eqref{eq:backcontr} with  the initial condition $z_0$. (b): Plot of $M_s$. (c): Plot of the feedback control function $u$.
	}
\label{fig:simulation1}
\end{figure}

\begin{remark}\label{see:convergremark}
Note that the feedback satisfies
 \begin{eqnarray}
      \sup_{\varepsilon \to 0}\{ \vert u (\mathcal{X})\vert: \mathcal{X}\in\mathcal{N}, \norm{X}_1\in \mathcal{B}(0,\varepsilon)\} &\longrightarrow 0.
 \end{eqnarray}
 The advantage of applying feedback
 control is that when the density of the target population
 decreases, the control also decreases.
\end{remark}
\begin{remark}
	It is important to note that the backstepping
	feedback control \eqref{eq:backcontr}
	does not depend on the environmental capacity
	$K$, which is also an interesting feature for
	 the field applications. In the case $K=+\infty$, the Eq \eqref{eq:SS11E1} becomes
	 \begin{gather}\label{eq:S11E11}
		\dot E = \beta_E F  -(\delta_E + \nu_E)E,
	 \end{gather}
	  and we prove by 
	 the same process that the same feedback law \eqref{eq:backcontr}, ensures the exponential stability of the SIT
	 system \eqref{eq:S11E11},\eqref{eq:S11E2}--\eqref{eq:S11E6} with the same lower bound of the exponential convergence rate.
\end{remark}
	 Our
	 stabilization result is the following one.
\begin{theorem}
	\label{thm-backsteppingKinfty}
	Assume that \eqref{eq:inequalitytheta} holds and $K=+\infty$.
	Then $\textbf{0}\in \mathcal{N}$ is globally exponentially stable in $\mathcal{N}$ for 
	system \eqref{eq:S11E11}, \eqref{eq:S11E2}--\eqref{eq:S11E6}  
	with the feedback law \eqref{eq:backcontr}. The exponential convergence 
	rate is bounded by $c_p>0$ defined in \eqref{cp}.
\end{theorem}
 \begin{remark}\label{see:Robustremark}
     Let us assume that the heterogeneity of the intervention zone strongly impacts the mating of female mosquitoes with sterile males more than we would have estimated. Suppose the estimated mating rate for the control \eqref{eq:backcontr} is $\eta_2^e=0.7$ and  let the mating rate is $\eta_2^r= 0.4$ for the dynamics. Keeping the other parameters and the same initial condition, we obtain the following figure.
 \begin{figure}[h]
	\centering
	\begin{subfigure}[H]{0.3\textwidth}
		\centering
		\includegraphics[width=\textwidth]{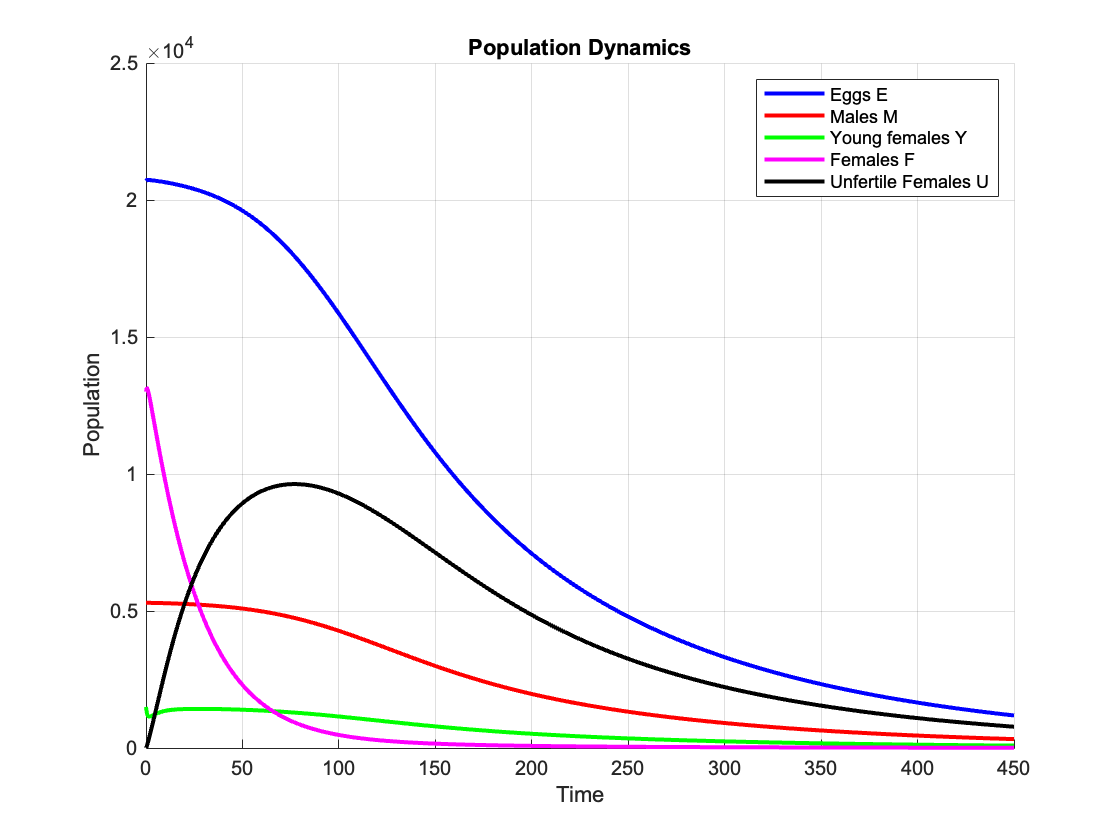}
	    \caption{Evolution of states $E$, $M$, $Y$, $F$ and $U$ } 
		\label{fig:evolutionEMFYU}
	\end{subfigure}
	\hfill
	\begin{subfigure}[H]{0.3\textwidth}
		\centering
	     \includegraphics[width=\textwidth]{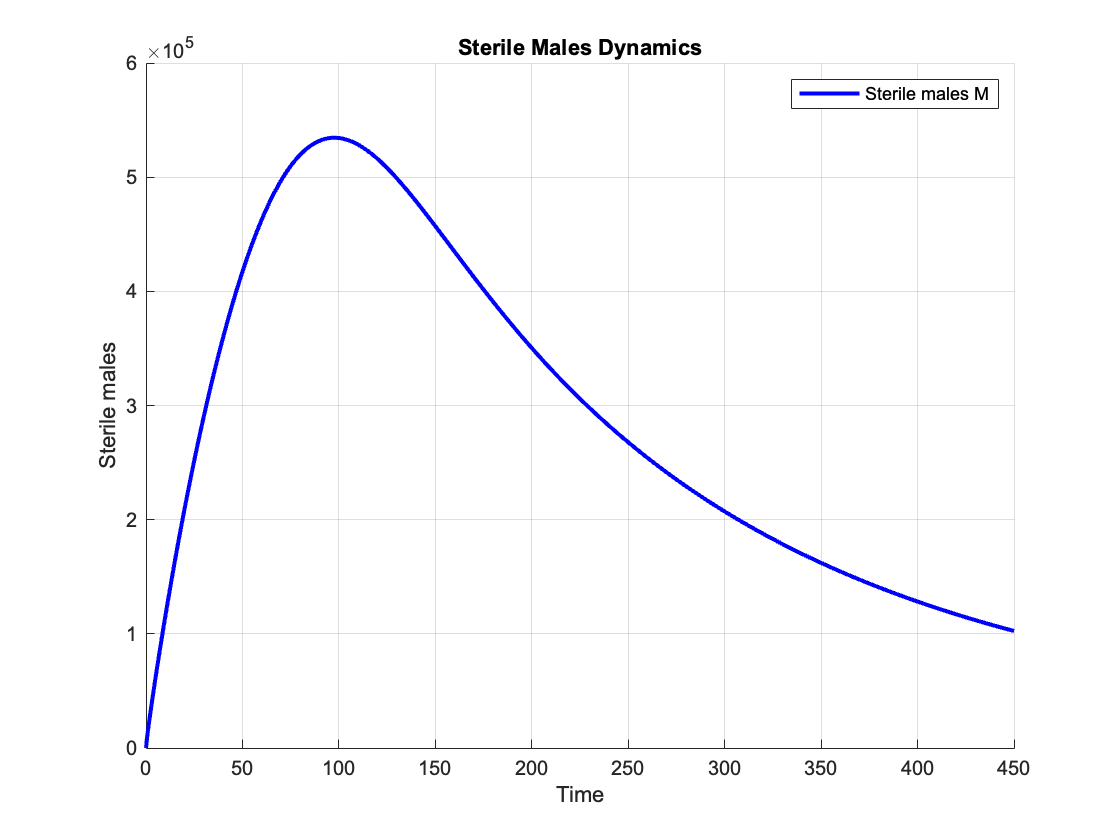}
	    \caption{Evolution of $M_s$ }
		\label{fig:EvolutionMs}
	\end{subfigure}
	\hfill
	\begin{subfigure}[H]{0.3\textwidth}
		\centering
		\includegraphics[width=\textwidth]{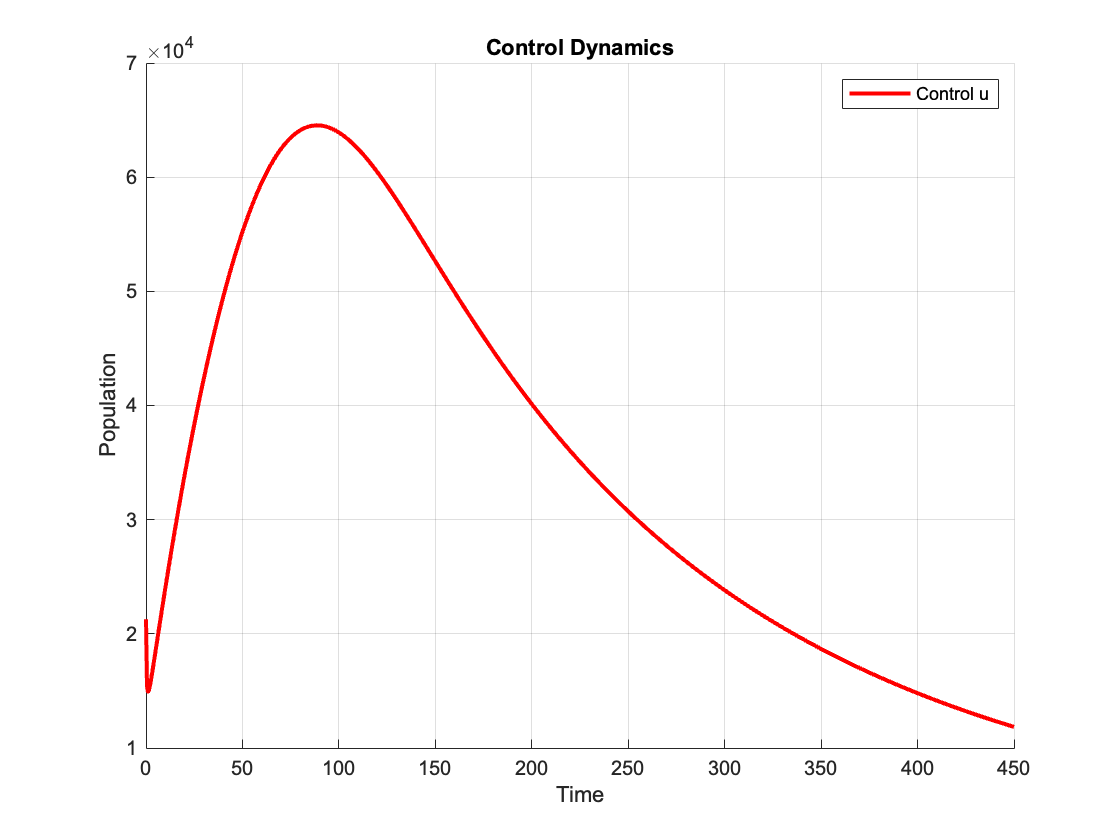}
	    \caption{Evolution of the control function $u$ }
		\label{fig:controletar}
	\end{subfigure}
	\vspace{2mm}
	\caption{(a): Plot of $E,M,Y$, $F$ and $U$ when applying the feedback \eqref{eq:backcontr} with $\eta_2^e = 0.7$ while $\eta_2^r = 0.4$ for the dynamics. (b): Plot of $M_s$. (c): Plot of the feedback control function $u$.
	}
\label{fig:simulationtest}
\end{figure}
This parameter considerably impacts the convergence time of the states of the system. Note that with $e = 3\times 10^{-1}$ of error difference, we still have convergence. Estimation errors of the order of $e = 10^{-2}$ will have a negligible impact on the convergence time. This is because the backstepping control also depends on the states of the system. Thus, the states make a correction that can compensate for a certain margin of error. Unlike control, which only depends on the parameters, estimation errors have no correction from the dynamics. Therefore, this can be fatal to the success of the intervention. In practice, many external factors impact the life cycle of mosquitoes. These factors modify parameters such as birth, hatching, and fertilization rates. These factors are, for example, rainfall and the topography of the region. A SIT model that can integrate these factors is challenging to study (see \cite{spatialSITwork}) . The success of  a SIT intervention depends strongly on the  robustness of the control strategy. The results of our previous test  that is reported in Figure \ref{fig:simulationtest}, show us the advantage feedback control can provide in terms of robustness.
 \end{remark}

\section{Observer design for SIT model}\label{see:Observersection}
The application of feedback control
requires measuring states such as eggs
$E$ and young females $Y$ of
the intervention zone over time.
In practice, it is always important
 to estimate the density
 of adult mosquitoes to
 intervene in an area.
 These data are collected
 using mosquito traps distributed
 throughout the region. Despite various
 technological advances to improve
 these traps, it should be noted that
  some data are still  easier
  to be measured than others. Measuring
   mosquito density in the aquatic phase $E$
    is difficult, specially in a
	heterogeneous area. It is also
	challenging to measure young
	females $Y$ because females come
	in three categories, and we need
	to distinguish between unfertilized
	and fertilized females. Males are more easily
	measured because they are distinguishable.
	It can  be also easy to distinguish wild males
	from laboratory males by marking processes
	applied to laboratory males.
	 In this part of
	  our paper, we will assume that the density
	  of wild males and that of sterile males
	  can be measured continuously. Our objective
	  is to estimate the other densities. Observer design
	  for nonlinear dynamic systems is a technique
	  used in control theory to estimate the states
	  of a system when only partial or indirect measurements
	  are available. The difficulties in dealing with observer
      problems for general  nonlinear systems is  the proof of
	  global convergence of the estimation error.
	  Much literature exists on state observers and
	  filters for nonlinear systems as they play crucial
	  roles in control theory. To simplify the nonlinearity $F(1-\frac{E}{K})$ of 
	  the SIT model, in this section we consider  the simplified SIT model for  environmental  capacity
	  $K=+\infty$. On the one hand, the reason for studying 
	  such a model is that the simplified model can be 
	  considered relevant from a biological point of
	   view within a large intervention domain or 
	   in areas where environmental capacity is difficult to estimate. 
	   On the other hand, based on the result presented in Theorem~\ref{thm-backsteppingKinfty},
	    the proposed feedback law \eqref{eq:backcontr} still stabilizes the simplified model 
	   around zero with the same convergence rate. 
	  We consider the following output control system:
	  \begin{align}
		&\dot E = \beta_E F  -(\delta_E + \nu_E)E,\label{eq:SO11E1}\\
		&\dot M = (1-\nu)\nu_E E - \delta_MM,\label{eq:SO11E2}\\
		&\dot Y = \nu\nu_E E - \frac{\Delta\eta M }{M+M_s} Y -(\eta_1+\delta_Y)Y,\label{eq:SO11E3}\\
		&\dot F = \frac{\eta_1 M }{M+M_s} Y -\delta_FF,\label{eq:SO11E4}\\
      &\dot U = \frac{\eta_2 M_s }{M+M_s}Y -\delta_UU,\label{eq:SO11E5}\\
		&\dot M_s = u -\mu_sM_s,\label{eq:SO11E6}\\
		& y_1 = M, \label{eq:SO11E7}\\
		&y_2 = M_s,\label{eq:SO11E8}
	\end{align}
where the states is $X = (E,M, Y, F,U,M_s)^T\in \mathcal{N}$, the control is $u\in[0,+\infty)$ and the output is $y = (M, M_s)^T\in\RR_+^2$.

In particular, in this model we are confronted with
	  a difficulty in which most observer construction
	  theories are invalid because of the singularity
	  at the origin. To go around this difficulty, we will
	   use the fact that the main nonlinearity term  $\frac{M}{M+M_s} $ is bounded
	  and essentially the most accessible data to measure.
	 This leads us to develop an observer for this type of system.
	\subsection{Observer design for a class of nonlinear systems}\label{see:Observertheory}
The usual observers for linear systems are the
Luenberger observer and the Kalman observer.
Observer design for a nonlinear
 system is a complex problem in control
 theory and has received much attention from many authors yielding a large literature of methods.
Among than, the most famous are the change
of coordinates to transform the nonlinear system into a
linear system
\cite{keller1987non,astolfi2006global,krener1983linearization,bernard2017,bookBoutatZheng}
 and a second approach consists in using the extended Kalman filter (EKF) \cite{boutayeb1999strong,reif1996linearisation,krenerarthur,Valibeygi}.
The state observer is called  an exponential
 state observer if the observer error
 converges exponentially to zero.
 In this section we provide an explicit construction of a
  global observer for  the following system.
	 \begin{equation}
		 \left\{\begin{aligned}
			 &\dot x(t) = A x(t) + B(y(t))x(t)+Du(t),\\& y(t) = Cx(t),
		 \end{aligned}
		 \right.\label{eq:MIMO}
	 \end{equation}
	 where $x(t)\in\RR^n$, is the state vector, $u(t)\in\RR^p$ is the input
	 vector and $y(t)\in\RR^m$ is the output vector. $A\in\RR^{n\times n}$ and $C\in\RR^{m\times n}$
	 are the appropriate matrices. The matrice $B(y(t))$ is in the form
	 \begin{gather}
		 B(y(t)) = \sum_{i,j=1}^{n,n}b_{ij}(y(t))e_n(i)e^T_n(j).
	 \end{gather}
	
		We assume that for all $y(t)\in\RR^m$ the coefficients $b_{ij}$ are bounded for all $i = 1,\cdots,n$ and $j=1,\cdots,n$ and denote
		\begin{align}\label{eq:maxofderivative}
			\overline{b}_{ij} = \max_{t}(b_{ij}(y(t)))\;\;\mbox{and}\;\; \underline{b}_{ij}=\min_{t}(b_{ij}(y(t))).
		\end{align}
		Then, the parameter vector $b(t)$ remains in a bounded convex domain $\mathcal{S}_{n,n}$ of which $2^{(n^2)}$ vertices are defined by:
		\begin{align*}
			\mathcal{V}_{\mathcal{S}_{n,n}} = \{\eta=(\eta_{11},\cdots,\eta_{1n},\cdots,\eta_{nn})| \eta_{ij}\in\{\underline{b}_{ij},\overline{b}_{ij}\}\}.
		\end{align*}
 A state observer corresponding to \eqref{eq:MIMO} is given as follows:
 \begin{equation}
	 \left\{\begin{aligned}
		 &\dot{\hat{x}}(t) = A \hat{x}(t) + B(y(t))\hat{x}+Du(t) - L(C\hat{x}-y(t)),\\& \hat{y}(t) = C\hat{x}(t),
	 \end{aligned}
	 \right.\label{eq:observerMIMO}
 \end{equation}
 where $\hat{x}(t)$ denotes the estimate of the state $x(t)$.
 The dynamics of the observer error $e(t):= \hat{x}(t)-x(t)$ is
	 $\dot{e}(t) = (A-LC)e(t) + B(y(t))e(t) = (A+B(y(t)) - LC) e(t).$
 We define
 \begin{gather}
	 \mathcal{A}(b(t)) = A + \sum_{i,j=1}^{n,n}b_{ij}(y(t))e_q(i)e^T_n(j).
 \end{gather}
 The dynamics of the observer error becomes
 \begin{gather}\label{eq:errorobserverdyn}
	 \dot{e}(t) =(\mathcal{A}(b(t)) - LC) e(t).
 \end{gather}
 The observation problem consists in finding a gain $L$ such that
 \eqref{eq:errorobserverdyn} converges exponentially
 towards zero. We use the following results in \cite{zemouche2005observer}.
 \begin{theorem} 
	\label{see:maintheorem}
	 The observer error converges exponentially towards zero if there exist matrices $P=P^T>0$ and $R$ of appropriate dimensions such
	 that following LMIs are feasible:
	 \begin{align}\label{eq:LMIcondition}
		 \mathcal{A}^T(\eta)P -C^TR+ P\mathcal{A}(\eta) -R^TC + \xi I<0,\\\forall\;\eta\in\mathcal{V}_{\mathcal{S}_{n,n}},
	 \end{align}
	 for some constant $\xi >0$.
 When these $LMIs$ are feasible, the observer gain $L$ is given by $L = P^{-1}R^T$.
 \end{theorem}
 {\bf Proof.}
 We follow \cite{zemouche2005observer} and  consider the following quadratic Lyapunov
 function
 \begin{equation}
	 \mathcal{V}(e) = e^TPe,
 \end{equation}
 where $P$ is the matrix in Theorem \ref{see:maintheorem}.
 We have
	 $\dot{\mathcal{V}}(e)(t) = e(t)^TF(b(t))e(t),$
 where
	 $F(b(t))= (\mathcal{A}(b(t)) -LC)^TP +P(\mathcal{A}(b(t)) -LC).$
 For $e(t)\not=0$ the condition $\mathcal{V}(e(t))>0$ is satisfied because $P>0$ and
 the condition $\dot{\mathcal{V}}(e(t))<0$ is satisfied if we have
 \begin{equation}
	 F(b(t))<0\;\;\mbox{for all}\;\;b(t)\in\mathcal{S}_{n,n}.
 \end{equation}
 Since the matrix function $F$ is affine in $b(t)$, using a convexity argument we deduce that $\forall \;t\geq 0$
 \begin{gather}\label{dotU<-U}
	 \dot{\mathcal{V}}(e(t))< - \xi\norm{e(t)}_P^2,
 \end{gather}
 if the following condition is satisfied
	  $F(\eta)< - \xi I,$ $\;\;\forall\;\eta\in {\mathcal{V}_{n,n}}.$
 Thus, if \eqref{eq:LMIcondition} holds, this inequality is also satisfied.
 \hfill $\Box$\\

 \subsection{Application to the SIT model}	
  We rewrite the output SIT models \eqref{eq:SO11E1}--\eqref{eq:SO11E7} as
 \begin{equation}
	 \left\{\begin{aligned}
		 &\dot X = A X + B(y)X+Du,\\& y = CX,
	 \end{aligned}
	 \right.\label{eq:SITMIMO}
 \end{equation}
 where  $X = (E,M,Y,F, U,M_s)^T$,
 \begin{align*}
	 &{A} = \left(\begin{array}{cccccc}-(\delta_E+\nu_E)&0&0&0&\beta_E&0\\(1-\nu)\nu_E&-\delta_M&0&0&0&0\\\nu\nu_E&0&-(\eta_2+\delta_Y)&0&0&0\\0&0&0&-\delta_F&0&0\\ 0&0&0&0&-\delta_U&0\\0&0&0&0&0&-\delta_s
	 \end{array}\right),\\&B(y) = \left(\begin{array}{cccccc}0&0&0&0&0&0\\0&0&0&0&0&0\\0&0&-\Delta\eta\frac{y_1}{y_1+y_2}&0&0&0\\0&0&\eta_1\frac{y_1}{y_1+y_2}&0&0&0\\ 0&0&\eta_2\frac{y_2}{y_1+y_2}&0&0&0\\0&0&0&0&0&0
	 \end{array}\right)\;\\& C = \left( \begin{array}{ccccccc} 0&1&0&0&0&0\\0&0&0&0&0&1
	 \end{array}
	 \right),\;\;\; D = (0,0,0,0,0,1)^T.
 \end{align*}
 As, $\mathcal{N}$  is an invariant set, one has
 $0\leq\frac{y_1}{y_1+y_2}\leq1$.
 Solving the corresponding equation of \eqref{eq:LMIcondition} with $\xi=1$
 in MATLAB, we get

 \begin{align}
	 &P = 10^4\left(\begin{array}{cccccc}
		 0.0219 &  -0.1567 &  -0.1531 &  -0.1703&   -0.0344&         0\\
   -0.1567&    8.9301&   -0.8472&   -0.8081&   -0.4929 &        0\\
   -0.1531&   -0.8472&    4.5716&    0.9277&    1.0845&         0\\
   -0.1703&   -0.8081&    0.9277 &   4.3088&   -2.3012 &        0\\
   -0.0344&   -0.4929&    1.0845 &  -2.3012&    4.7413&         0\\
         0 &        0  &       0 &        0 &        0 &   3.7267
		 \end{array}\right),\;
	 \\&
	 R= 10^3\left(\begin{array}{cccccc}
		0.2352&    0.9704  & -0.4415&   -1.1401 &   0.0690&         0\\
         0 &        0  &       0   &      0    &     0  &  1.4162
		 \end{array}\right),
	 \end{align}
	 \begin{align}\label{gainmatrixL}
		L = \left(\begin{array}{cc}
			50.6342 &        0\\
    1.4150&         0\\
    0.9426  &       0\\
    2.6547 &        0\\
    1.6023  &       0\\
         0  &  0.3800
		 \end{array}\right).
	 \end{align}
	 \begin{figure}[H]
		\centering
		\begin{subfigure}[H]{0.45\textwidth}
			\centering
			\includegraphics[width=\textwidth]{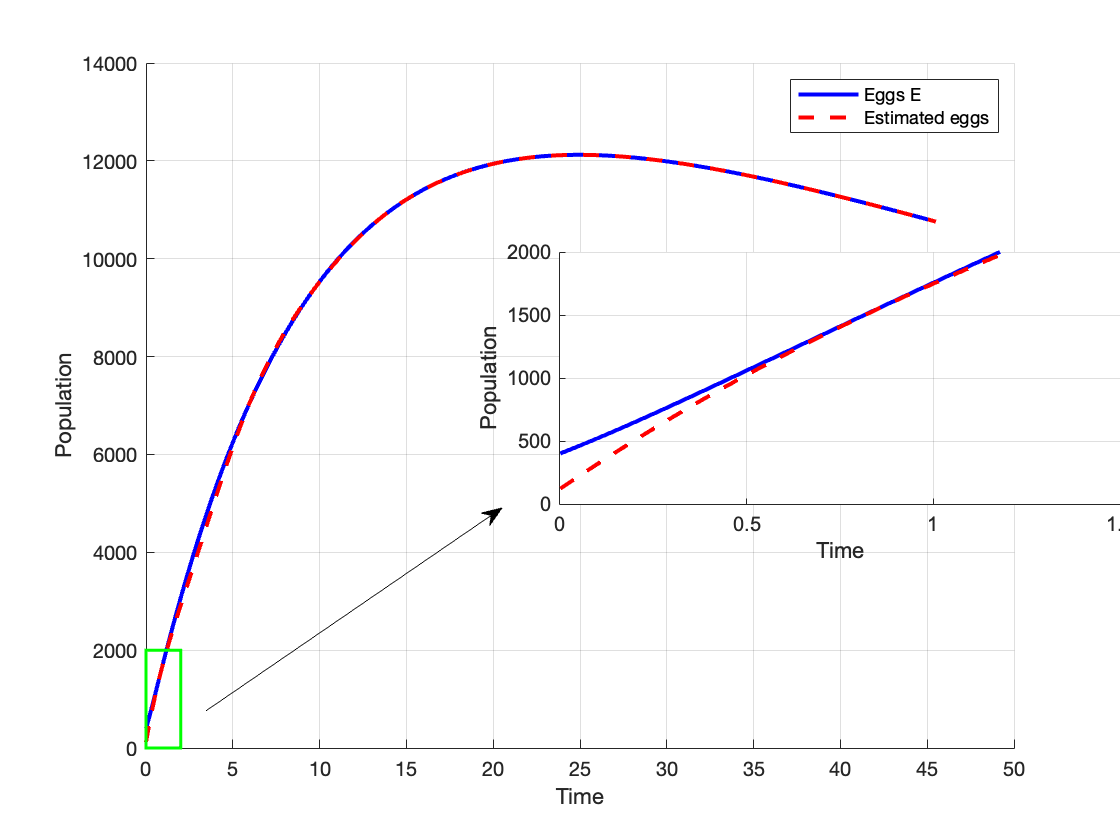}
			\caption{Evolution of states $E$ and corresponding  $\hat{E}$  }
		\end{subfigure}
		\hfill
		\begin{subfigure}[H]{0.45\textwidth}
			\centering
			 \includegraphics[width=\textwidth]{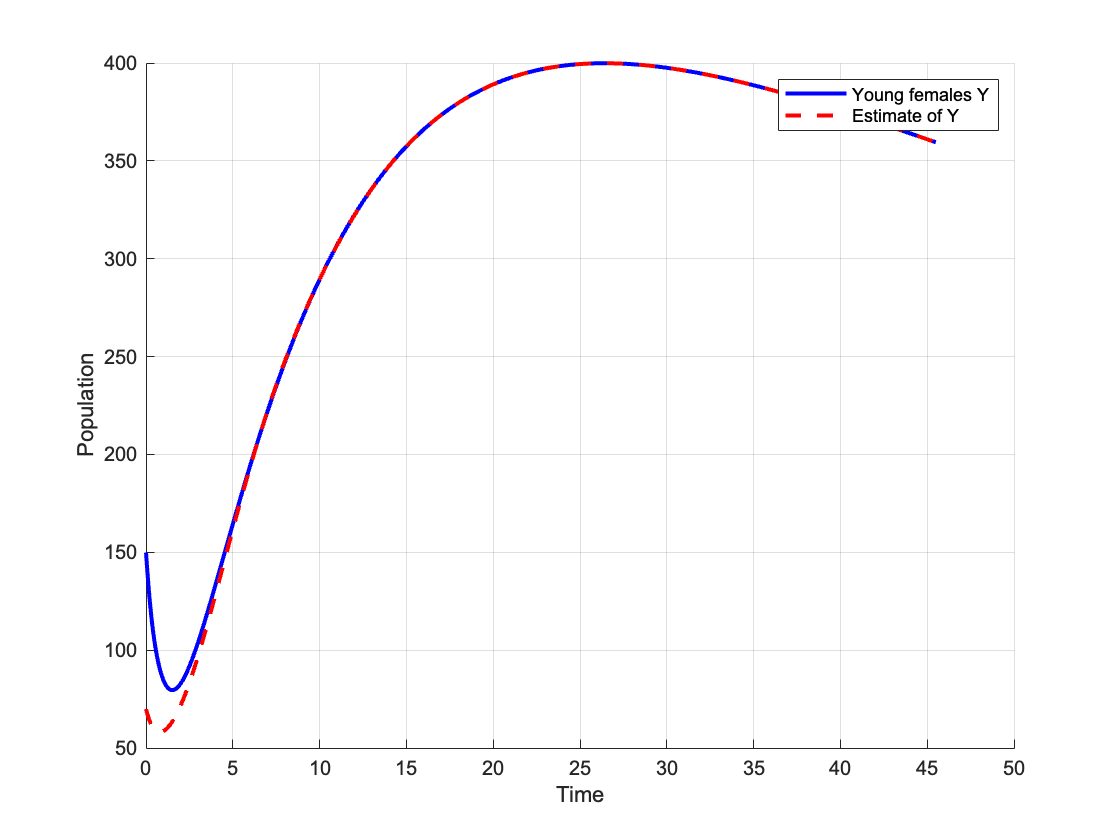}
			\caption{Evolution of states $Y$ and corresponding  $\hat{Y}$ }
		\end{subfigure}
		\hfill
		\begin{subfigure}[H]{0.45\textwidth}
			\centering
			\includegraphics[width=\textwidth]{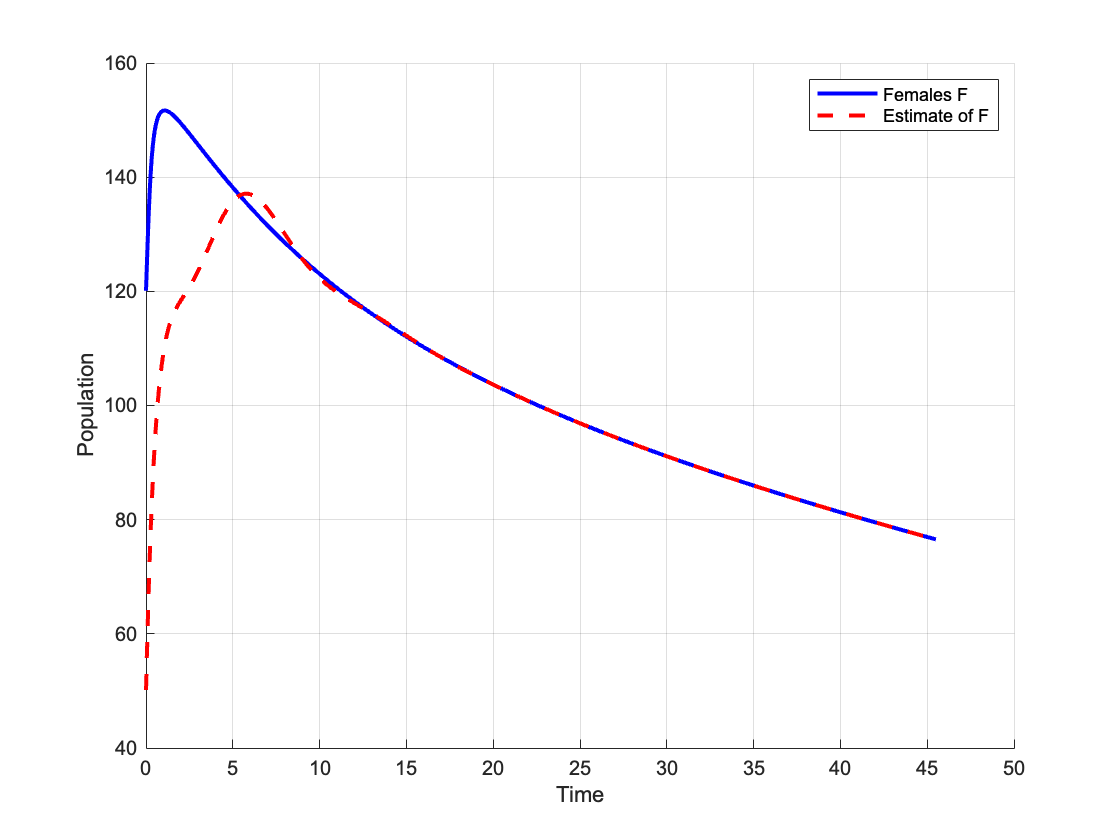}
			\caption{Evolution of states $F$ and corresponding  $\hat{F}$ }
		\end{subfigure}
  \begin{subfigure}[H]{0.45\textwidth}
			\centering
			\includegraphics[width=\textwidth]{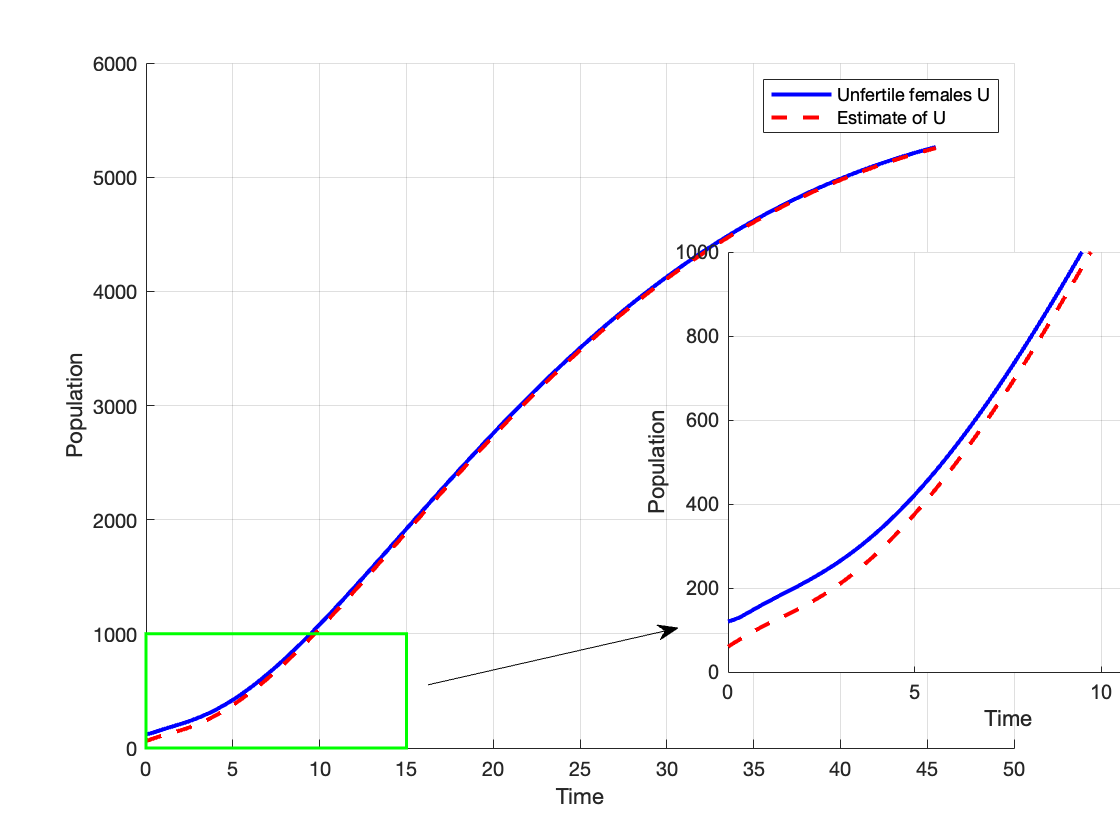}
			\caption{Evolution of states $U$ and corresponding  $\hat{U}$ }
		\end{subfigure}
		\vspace{2mm}
		\caption{Simulation of the system and the observer under the constant input. The dashed lines illustrate the state estimate curves }\label{fig:obsersystem1}
	\end{figure}
	
	With the parameters given in table \ref{eq:tableparametre}, the result for a simulation run 
	with  $x_0  = ( 400, 100, 150, 120,120,50)^T$,  $\hat{x}_0 = (120, 70, 70, 50,60,0)^T$  
	and $u = 500000$ is plotted  in Figure \ref{fig:obsersystem1}. 
	The asymptotic behavior of the different estimates $\hat E, \hat F, \hat Y$ and $\hat U$ (dashed), illustrates the exponential 
	 convergence of the estimation error show in 
	  Theorem~\ref{see:maintheorem}.

	 \section{Dynamic output feedback }\label{see:feed-obser}
	 The  feedback control \eqref{eq:backcontr}
	  depends on the states $E$, $M$, $Y$ and $M_s$.
	 From the measurement of states $M$ and $M_s$, an observer system has been built
	 in the  previous section. This state observer is used to estimate both
	 eggs $E$ and young  females $Y$. In this section we show that ${\bf u}(\hat{X},y)$
	 stabilizes the dynamics at the origin. We consider the coupled system
	 \begin{equation}\label{eq:coupledsystem}
		 \left\{\begin{aligned}
			 &\dot X = f(X, {\bf \hat{u}}(\hat{X},y)),\\
		 & \dot{\hat{X}} = f(\hat{X},{\bf \hat{u}}(\hat{X},y)) - L(C\hat{x}-y),
		 \end{aligned}
		 \right.
	 \end{equation}
	 with
	 \begin{gather}\label{eq:observercontrolU}
		 {\bf \hat{u}}(\hat{X},y) = \max\left(0,S(\hat{X}, y )\right).
	 \end{gather}
	 where
	 $S: \RR^4\times\RR_+^2\rightarrow \RR$, $(\hat{X},y)^T\mapsto S(\hat{X},y)$ is defined by
	 \begin{gather}\label{eq:Sdefine}
		 S(\hat{X},y):=	G(\hat{E}, M,\hat{Y}, M_s )
	 \end{gather}
	
	 The main result of this section is the following Theorem.
	
	 \begin{theorem}
		 \label{see:StabilityObsheorem}
		 Assume that (\ref{eq:inequalitytheta}) holds. Then
		 $\textbf{0}\in \mathcal{E}=\mathcal{N}\times \RR^6$ is globally
		 exponentially  stable in $\mathcal{E}$
		 for system \eqref{eq:coupledsystem}
		 with the feedback law \eqref{eq:observercontrolU}. The convergence rate
		 is bounded by the positive constant $c_e$ defined by
		 \begin{gather}\label{ce}
			c_e:=\min\{c_1,c_2,c',\frac{\xi}{4}\}.
		 \end{gather}
	 \end{theorem}
\begin{proof}
	 Let $\lambda>0$ and we define $H : \mathcal{E}\to\RR$ by
\begin{gather}
	H(X,\hat{X}) = W(X) + \lambda\sqrt{\mathcal{V}(e)}
\end{gather}
with  $ e = \hat{X}-X$.
\begin{gather}
H \text{ is continuous on}\;\;  \mathcal{E} \;\;\mbox{and}\;\;\mathcal{C}^1
\;\; \mbox{on}\;\;\text{ $\mathcal{E}\setminus \left\{(X,\hat{X})\in \mathcal{E};\; M+M_s=0\right\}$},
\\
\label{Winftyatinfty2}
H(X,\hat{X})\to +\infty \;\;\mbox{as} \;\; \norm{(X,\hat{X})} \to +\infty,
\\
\label{H>0-sec2}
H(X,\hat{X})>H(\textbf{0})=0,\; \forall (X,\hat{X})\in \mathcal{E}\setminus\{\textbf{0}\}.
\end{gather}
In this proof, from now on  we assume that $(X,\hat{X})^T$ is in $\mathcal{E}$. Until \eqref{dotH<0} included, we also assume that
\begin{gather}
\label{MMsnotboth0}
(M,M_s)\not =(0,0).
\end{gather}
One has
\begin{eqnarray*}	
	\dot{H}(X,\hat{X})=
	\dot{V}(X) + \alpha\frac{(\theta M-M_s)}{(\theta M + M_s)^2}\Big[ \frac{\phi Y(\theta M + M_s)^2}{\alpha(M+M_s)} \\+ ((1-\nu)\nu_E\theta E -\theta \delta_M M)(\theta M +3M_s)\nonumber
	\\-{\bf \hat{u}}(\hat{X},y)(3\theta M + M_s)+\delta_s{M}_s(3\theta M + M_s)\Big] + \lambda\frac{\dot{\mathcal{V}}(e)}{2\sqrt{\mathcal{V}(e)}}.
	\label{dotW=}
	\end{eqnarray*}
Replacing the term ${\bf \hat{u}}(\hat{X},y)$ by ${\bf \hat{u}}(\hat{X},y)-{\bf \hat{u}}({X},y) + {\bf \hat{u}}({X},y)$
we get
  \begin{gather}	
	\dot{H}(X,\hat{X})=
	\dot{W}(X)+ \alpha\frac{(\theta M-M_s)(3\theta M + M_s)}{(\theta M + M_s)^2}({\bf u}(\hat{X},y)- {\bf u}(X,y)) + \lambda\frac{\dot{\mathcal{V}}(e)}{2\sqrt{\mathcal{V}(e)}}.
	\end{gather}

\begin{lemma}
	There exist $C>0$ such that,  for all $(X,\hat{X})\in\mathcal{E}$ and for all $y\in\RR_+^2$,
	\begin{gather}
		\norm{{\bf \hat{u}}(\hat{X},y) - {\bf \hat{u}}(X,y)} \leq C\norm{\hat{X}-X}.
	\end{gather}
\end{lemma}
Note that $\dot{\mathcal{V}}(e)\leq -  \xi\norm{e}_P^2$. Thanks to this lemma, there exists $C'>0$ independent of $y$ such that
\begin{align}	
	\dot{H}(X,\hat{X})&\leq  \dot{W}(X)+ C'\norm{e} - \xi\lambda\frac{\norm{e}_P^2}{2\sqrt{\mathcal{V}(e)}}.
	\end{align}
 Note that there exists a constant $\beta>0$ such that
 and $ \norm{e}\leq \beta\norm{e}_P$. So
\begin{align}	
	\dot{H}(X,\hat{X})&\leq
	\dot{W}(X) - (\frac{\lambda\xi}{2}-\beta C')\norm{e}_P.
\end{align}
Hence for $\lambda =4C'\beta/\xi$, and using the relation \eqref{c1} and \eqref{c2},
\begin{align}	
	\dot{H}(X,\hat{X})\leq
	-\min\{c_1,c_2\} W(X) - \frac{\lambda\xi}{4}\norm{e}_P.
\end{align}
We conclude that there exists
a constant
\begin{gather}\label{cs}
	c_s := \min\{c_1,c_2,\frac{\xi}{4}\}
\end{gather}
such that
\begin{gather}\label{dotH<0}
	\dot H(X,\hat{X}) < -c_s H(X, \hat{X}),\;\mbox{if} \;M+M_s\not=0.
\end{gather}
Let us now deal with the case where \eqref{MMsnotboth01} is not satisfied. As we explained previously in the proof of the Theorem~\ref{thm-backstepping}, it is sufficient to study  only the case $t_s\in (0,+\infty)$.
Let  $t\mapsto (E(t),M(t),Y(t),F(t), U(t), M_s(t),\hat{E}(t),\hat{M}(t),\hat{Y}(t),\hat{F}(t), \hat{U}(t), \hat{M}_s(t))^T$  be a solution (in the Filippov sense)
of the closed-loop system \eqref{eq:coupledsystem} such that, for some  $t_s\in (0,+\infty)$
\begin{equation}
\label{M+Ms=0tscouplesys}
M(t)+M_s(t)=0\; \forall t\in [0,t_s]
\end{equation}
Note that \eqref{M+Ms=0tscouplesys} implies that
\begin{equation}
\label{M=Ms=0tscouplesys}
M(t)=M_s(t)=0,\; \forall t\in [0,t_s]
\end{equation}
From \eqref{ineqMMs}, \eqref{M=Ms=0ts} and the definition of a Filippov solution, one has on $(0,t_s)$
   \begin{align}
    \label{eqreal-Filippov}
       & \left(\begin{array}{ccccc}
\dot{E}\\\dot{{M}}\\\dot{{Y}}\\\dot{{F}}\\\dot{{U}}\\\dot{{M}}_s
		\end{array}\right) = \begin{pmatrix}
			\beta_E {F}(1-\frac{E}{K}) - \big( \nu_E + \delta_E \big) {E}
			\\  (1-\nu)\nu_E E - \delta_M M
			\\\nu\nu_E E-\kappa(t)\Delta\eta Y-(\eta_2+\delta_Y)Y
			\\  \eta_1 Y \kappa(t) - \delta_F {F}
             \\ \eta_2(1-\kappa(t)) Y - \delta_U U
			\\ \max(0,\hat{Y}g_1+\hat{E}g_2) -\delta_s{M}_s
		\end{pmatrix}
   \\
   &\left(\begin{array}{cccccc} \label{eqhat-Filippov}
\dot{\hat{E}}\\\dot{\hat{M}}\\\dot{\hat{Y}}\\\dot{\hat{F}}\\\dot{\hat{U}}\\\dot{\hat{M}}_s
		\end{array}\right) = 
		    \begin{pmatrix}
			\beta_E \hat{F} - \big( \nu_E + \delta_E \big) \hat{E}
			\\  (1-\nu)\nu_E \hat{E} - \delta_M \hat{M}
			\\\nu\nu_E \hat{E}-\kappa(t)\Delta\eta \hat{Y}-(\eta_2+\delta_Y)\hat{Y}
			\\  \eta_1 \hat{Y} \kappa(t) - \delta_F \hat{F}
             \\ \eta_2(1-\kappa(t)) \hat{Y} - \delta_U \hat{U}
			\\ \max(0,\hat{Y}g_1+\hat{E}g_2) -\delta_s\hat{M}_s
		\end{pmatrix} -LC\hat{X},
    \end{align}
    with
\begin{equation}
\kappa(t)\in [0,1],\; g_1(t)\in \frac{\phi}{\alpha}[0,3\theta+1] \text{ and } g_2(t)\in
	(1-\nu)\nu_E\theta [0,4].
\end{equation}
From \eqref{M=Ms=0tscouplesys} and the second line of \eqref{eqreal-Filippov}, one has
\begin{equation}
\label{E=0tscouplsys}
E(t)=0, \; \forall t \in [0,t_s]
\end{equation}
 From the first line of \eqref{eqreal-Filippov} and  \eqref{E=0tscouplsys}, we get 
\begin{equation}
\label{F=0tscouplsys}
F(t)=0, \; \forall t \in [0,t_s].
\end{equation}
In the case where $Y(0) = 0$,  from the third line of \eqref{eqreal-Filippov} and \eqref{E=0tscouplsys}, one has
\begin{equation}
\label{Y=0tscouplsys}
Y(t)=0, \; \forall t \in [0,t_s].
\end{equation} 
To summarize, from \eqref{M=Ms=0tscouplesys},  the fifth line of \eqref{eqreal-Filippov}, \eqref{E=0tscouplsys}, \eqref{F=0tscouplsys} and \eqref{Y=0tscouplsys}
\begin{equation}
E(t)=M(t)=Y(t)=F(t)=M_s(t)=0 \text{ and }\dot U(t)=-\delta_UU(t),\;\forall t \in [0,t_s],
\end{equation} 
which, with \eqref{def-V}, \eqref{def-c>0} and \eqref{eq:lyapunov-functionW-M=Ms=0}, gives
\begin{gather}
\label{dotWleq-cW2}
\dot W(t)=-\sigma\delta_UU(t)\leq -\delta_UW(t), \; \forall t \in [0,t_s].
\end{gather}

In the case where $Y(0) >0 $,  from the third line of \eqref{eqreal-Filippov},
\begin{equation}
\label{Y>0tscouplsys}
Y(t)>0, \; \forall t \in [0,t_s],
\end{equation} 
which, together with the fourth line of \eqref{eqreal-Filippov} and \eqref{F=0tscouplsys}, implies
\begin{equation}
\label{kappa0tscouplsys}
\kappa(t)=0, \; \forall t \in [0,t_s].
\end{equation} 
Referring to this case already studied in the proof of Theorem~\ref{thm-backstepping} we get
\begin{gather}
\dot W(t)\leq -c'W(t), \; \forall t \in [0,t_s].
\end{gather}

\begin{align}
    \kappa(t)\in[0,1], g_1(t)\in \frac{\phi}{\alpha}[0,3\theta+1] \text{ and } g_2(t)\in
	(1-\nu)\nu_E\theta [0,4],\\\dot M_s(t) = \max(0,\hat{Y}g_1+\hat{E}g_2) -\delta_s{M}_s
\end{align}
Since $M_s(t) = 0 \;\forall t\in[0,t_s]$, $\max(0,\hat{Y}g_1+\hat{E}g_2)=0$.
For all $\kappa(t)\in [0,1]$, in these two cases, the dynamics of the observation error remains
\begin{gather}
	\dot{e} = (\mathcal{A}(\kappa(t)) - LC)e,
\end{gather}
    and one has
    \begin{gather}
        \dot{H}(X,\hat{X}) =-c^{\prime}W(X)-\frac{\lambda\xi}{2}\norm{e}_P.
    \end{gather}
	We conclude that there exists
	a constant
	\begin{gather}\label{cw}
		c_w := \min\{c',\frac{\xi}{2}\},
	\end{gather} 
	 such that
 \begin{gather}
    \dot{H}(X,\hat{X}) \leq -c_w H(X,\hat{X}).
 \end{gather}
 This proves  Theorem~\ref{see:StabilityObsheorem} and gives the global exponential stability
 with the exponential decay rate $c_e$ given by  relation \eqref{ce}.
 \end{proof}
	 \subsection{Numerical simulations}
	 We apply the backstepping control $u$ function
	 of the measured states $y$ and
	 the estimated states $\hat{E}$ and $\hat{Y}$ given by
	 the relation \eqref{eq:observercontrolU} with  the following  initial condition $x_0  = ( 20000, 5000, 1500, 12000,500)$ and $\hat{x}_0 = (2000, 500, 150, 1200,0)$.
		 \begin{figure}[H]
			\centering
		\begin{subfigure}[H]{0.4\textwidth}
			\centering
			\includegraphics[width=\textwidth]{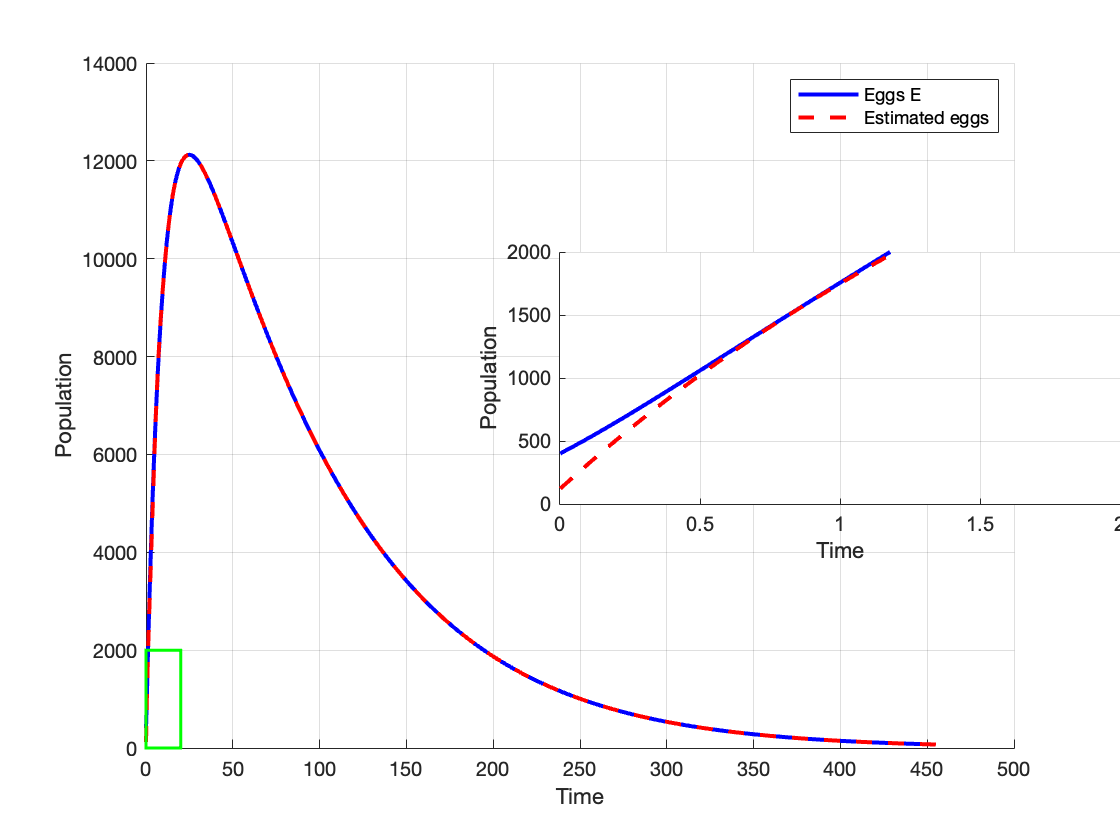}
			\caption{Evolution of states $E$ and estimate  $\hat{E}$  }
		\end{subfigure}
		\hfill
		\begin{subfigure}[H]{0.4\textwidth}
			\centering
			 \includegraphics[width=\textwidth]{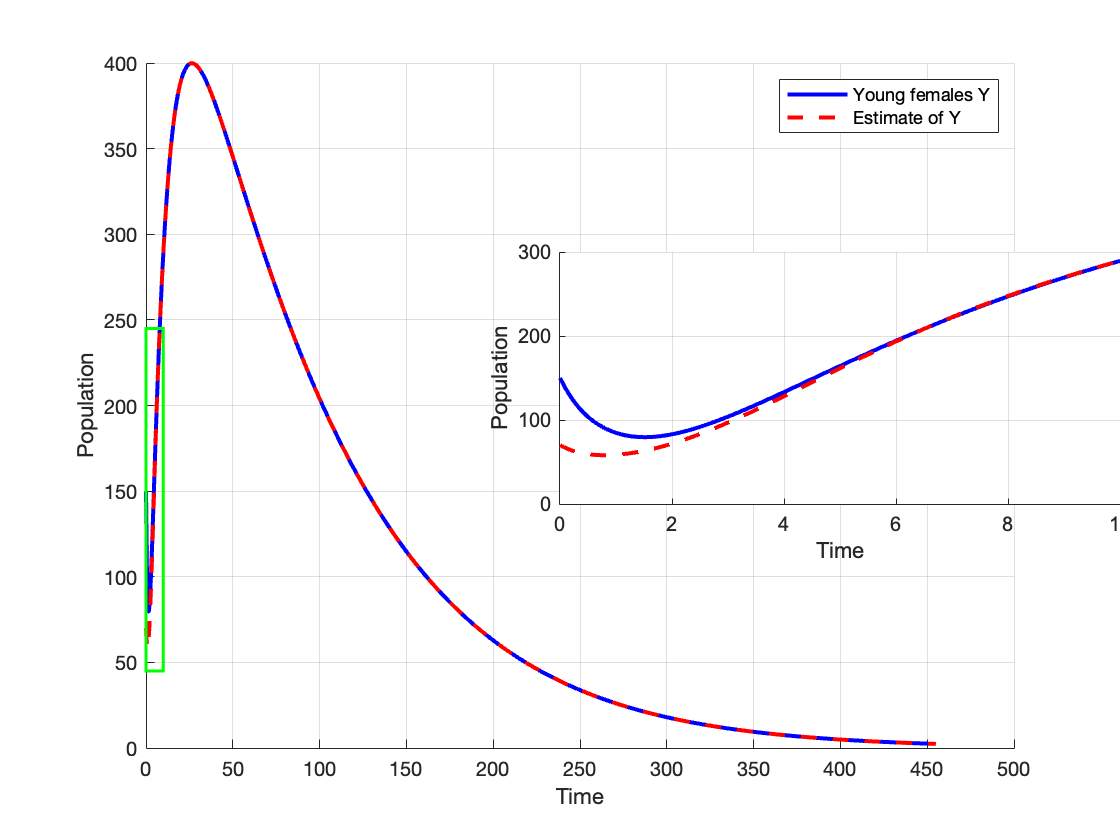}
			\caption{Evolution of states $Y$ and estimate  $\hat{Y}$ }
		\end{subfigure}
		\hfill
		\begin{subfigure}[H]{0.4\textwidth}
			\centering
			\includegraphics[width=\textwidth]{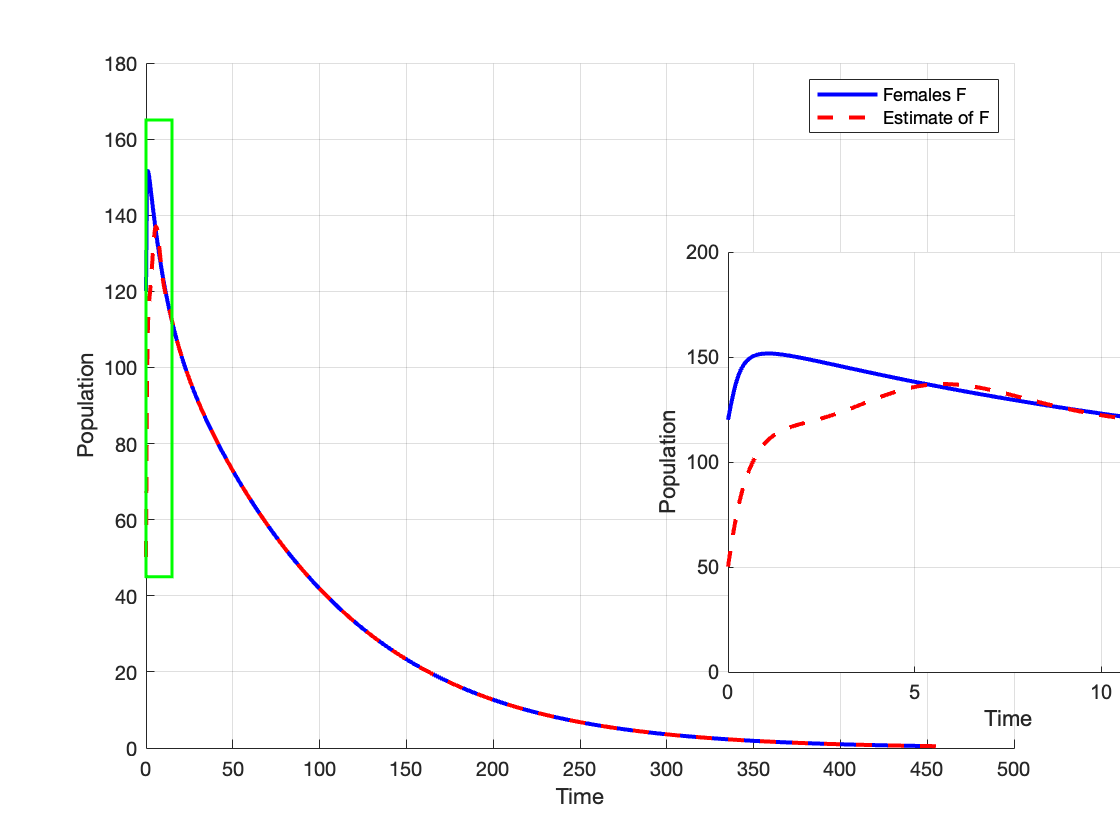}
			\caption{Evolution of states $F$ and estimate  $\hat{F}$ }
		\end{subfigure}
  \hfill
  \begin{subfigure}[H]{0.4\textwidth}
			\centering
			\includegraphics[width=\textwidth]{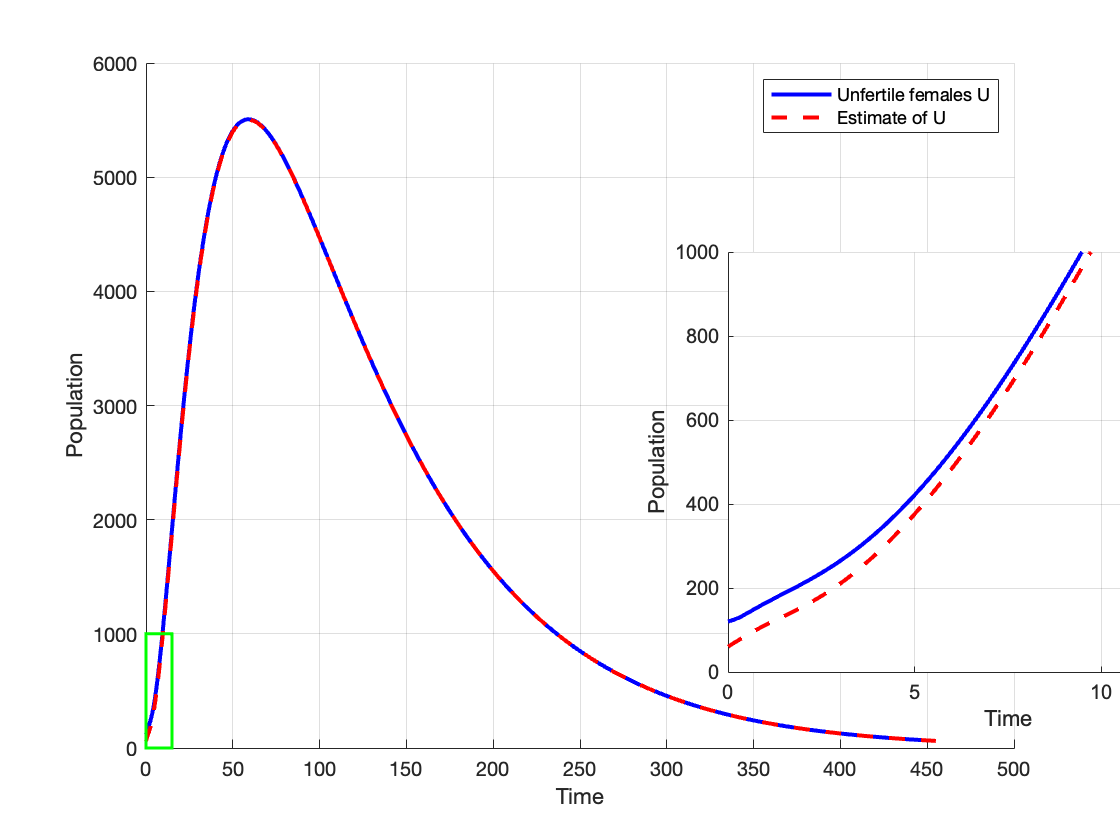}
			\caption{Evolution of states $U$ and estimate  $\hat{U}$ }
		\end{subfigure}
		\vspace{3mm}
			\caption{Simulation of the
			SIT model when applying backstepping feedback law with estimated and measured
			states \eqref{eq:observercontrolU} . }\label{fig:Couplesystem}
		\end{figure}
		\begin{figure}[H]
			\centering
			\includegraphics[width=0.6\textwidth]{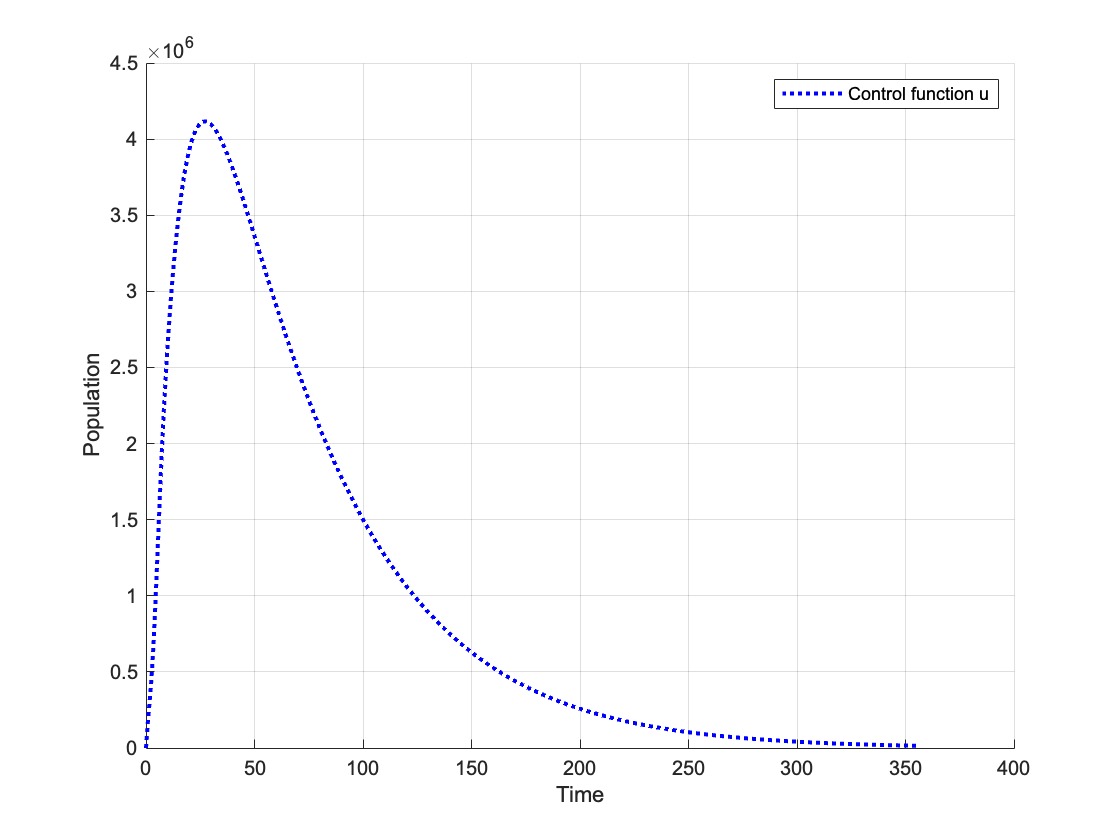}
				\caption{Evolution of control function ${\bf u}(\hat{X},y)$.}
		\label{fig:estimU}
		\end{figure}		
		The response of  system ~\eqref{eq:coupledsystem} to 
		the backstepping control \eqref{eq:observercontrolU} is 
		illustrated in the Figure \ref{fig:Couplesystem}. The 
		asymptotic convergence in large $t$ of the different estimated
		$\hat E, \hat F, \hat Y$ and $\hat U$  (dashed) 
		to their corresponding state variables $E, F, Y$ and $U$ respectively,  illustrates the exponential 
		convergence of the estimation error stated 
		in  Theorem~\ref{see:maintheorem}. The convergence 
		of the states and their estimates,  
		towards zero, proves the efficiency of the
		 control\eqref{eq:observercontrolU} as was show 
		 in the Theorem~\ref{see:StabilityObsheorem}. 
		 Figure \ref{fig:estimU} shows that the 
		  applied  control function decreases, when the density of the 
		  target population decreases.

	 \section{Conclusion}
  In this work, we have built a feedback control law 
  to stabilize the SIT model presented
   in \cite{esteva2005mathematical,anguelov2012mathematical} 
   at extinction. Control by state feedback is a type of
    control rarely proposed in the literature for 
	the overall stabilization of the SIT model.  
	The feedback control \eqref{eq:backcontr}  developed in this work has many advantages 
	including  robustness 
	to changing parameters.
  We have shown in Remark~\ref{see:Robustremark} 
  that despite the margin of error that can be made in 
  the estimation of the  parameters, 
  this feedback control  still makes the system converge to extinction.
  Moreover its does not 
  depend on environmental capacity and 
  this control law ensures exponential 
  stability with the same convergence rate  for the  SIT 
  system even  in the high  environmental capacity  limit (see Theorem~\ref{thm-backsteppingKinfty}).  
   Remark~\ref{see:convergremark} shows that when the density of the target population
  decreases, the control also decreases

   In Section~\ref{see:Observersection} of our work, we built an observer for the SIT model where, using the measurement of male mosquitoes, our state estimator gives us an estimate of the other states of the system. This aspect is rarely studied for this type of dynamics. An accurate estimate of the mosquito population enables resources to be allocated more efficiently. If the intervention is effective in some areas but not in others, resources can be reallocated to maximize impact. On the other hand, the data collected during the SIT intervention provides essential information on the impact of the control in the conditions of the intervention area. This will enable informed decisions on future control strategies to be adopted according to conditions in the intervention zone by adding complementary methods or by  adapting existing approaches.

   One of the applications we made was to show in Section~\ref{see:feed-obser} that by using the data estimated via our observer to adjust the feedback control, we globally stabilize the system upon extinction.
   Figure~\ref{fig:Couplesystem} shows that the
 difficulty of estimating eggs and young females during an intervention can be  compensated  by the application of
 the observer system.
   Data collected on the mosquito population is also used in epidemic prevention programs. They help to adapt public health programs for better control of mosquito-borne diseases.
   
\section*{Use of AI tools declaration}
The authors declare they have not used Artificial Intelligence (AI) tools in the creation of this article.

\section*{Acknowledgements}
The author wishes to thank Luis Almeida and Jean-Michel Coron for having drawn his attention to this problem and for the many  enlightening discussions during this work.

\section*{Conflict of interest}

The authors declare there is no conflict of interest.

\end{document}